%% file: KH.tex
\documentclass{article}
\usepackage{graphicx} 

\title{\textbf{Holomorphic Deformations of Compact Kähler Hyperbolic Manifolds}}
\author{ABDELOUAHAB KHELIFATI}
\date{}
\usepackage{graphicx} 
\usepackage{booktabs}
\usepackage[english]{babel}
\usepackage{amsmath}
\usepackage{amssymb}
\usepackage{amsfonts}
\usepackage{amsthm}
\usepackage{mathrsfs}
\usepackage{enumitem}
\usepackage{extarrows}
\usepackage{musicography}
\usepackage[utf8]{inputenc}
\usepackage[T1]{fontenc}
\usepackage[parfill]{parskip}
\usepackage{hyperref}
\hypersetup{colorlinks=true, linkcolor=blue, citecolor=red, urlcolor=blue
}
\usepackage{lipsum}
\usepackage{mathtools}
\usepackage[all]{xy}

\usepackage{makeidx}

\renewcommand{\thesubsection}{\arabic{section}.\arabic{subsection}}
\usepackage{fullpage}
\usepackage{amssymb, amsmath, amsthm, amscd}
\usepackage{titlesec}
\titleformat{\subsection}
  {\normalfont\Large\bfseries}
  {\thesubsection}{1em}{}
\numberwithin{equation}{section}

\theoremstyle{plain}
\newtheorem{theorem}{Theorem}[section]
\newtheorem{dfn}[theorem]{Definition}

\newtheorem{thm}[theorem]{Theorem}
\newtheorem*{mthm}{Main Theorem}
\newtheorem{prop}[theorem]{Proposition}
\newtheorem{lem}[theorem]{Lemma}
\newtheorem{cor}[theorem]{Corollary}
\newtheorem{conj}[theorem]{Conjecture}
\theoremstyle{definition}

\newtheorem{qst}[theorem]{Question}
\newtheorem{rmk}[theorem]{Remark}
\newtheorem{ex}[theorem]{Examples}
\newtheorem*{pr1}{\textit{Proof of Theorem \ref{58}}}
\newcommand{\B}{\mathbb{B}}
\newcommand{\C}{\mathbb{C}}
\newcommand{\D}{\mathbb{D}}
\newcommand{\CP}{\mathbb{P}}
\newcommand{\R}{\mathbb{R}}
\newcommand{\Sp}{\mathbb{S}}
\newcommand{\N}{\mathbb{N}}
\newcommand{\Z}{\mathbb{Z}}

\begin{document}

\maketitle

\vspace{5ex}

\large{\textbf{Abstract.}} The goal of this paper is to study the deformations of compact Kähler hyperbolic manifolds. We will propose slightly modified versions of Kähler hyperbolicity as a tool to give a first step towards investigating the deformation openness of Gromov's classical notion of Kähler hyperbolicity.
\tableofcontents
\section{Introduction}
The first notion of complex hyperbolicity was introduced by S. Kobayashi in 1967. It can be seen as a way of generalizing the well-known Schwarz-Pick lemma in complex analysis to complex manifolds (or complex spaces) by defining a special pseudometric attached to the manifold in question, called the \textit{Kobayashi pseudometric}. When this pseudometric is a genuine metric, the manifold is said to be \textit{Kobayashi hyperbolic} \cite{kobayashi1967invariant}. In the compact setting, R. Brody showed in \cite{brody1978compact} that Kobayashi hyperbolicity is equivalent to the non-existence of non-constant entire curves. Subsequently, M. Gromov introduced in one of his seminal papers (\cite{gromov1991kahler}), a form of hyperbolicity on complex manifolds in terms of Kähler metrics. There, he defined the notion of \textit{Kähler hyperbolicity}, "probably" inspired by the property of \textit{openness at infinity} introduced by D. Sullivan in \cite{sullivan1976cycles} to give a condition for the existence of an invariant transversal measure on the leaves of a foliation on a compact complex manifold. Several notions of hyperbolicity have been developed recently by generalizing the non-existence of non-constant entire curves to ruling out non-degenerate holomorphic maps from some $\C^k$ (with $2\leq k\leq\dim_\C X-1$) into the given complex manifold $X$ with some growth condition, or by replacing the Kähler class with other special metrics. We can cite a few : \cite{marouani2023balanced}, \cite{marouani2023skt}, \cite{bei2024weakly}, \cite{ma2024strongly}, \cite{kasuya2025partially}, \cite{kasuya2025higher}.

Recall that a compact complex manifold $X$ is called :
\begin{itemize}
    \item\cite{gromov1991kahler}. \textbf{Kähler hyperbolic} if it admits a Kähler metric $\omega$ which is $\widetilde{d}$-bounded. 
    \item\cite{marouani2023skt}. \textbf{SKT hyperbolic} if it admits an SKT metric $\omega$ which is $(\widetilde{\partial+\bar\partial})$-bounded.
\end{itemize}

The main purpose of this article is to take a first step towards the study of the deformation openness of Kähler hyperbolicity by introducing two slightly modified versions. A first one called \textbf{Strong Kähler hyperbolicity}, this is an hyperbolicity notion pinched right between \textit{Real hyperbolicity} on Kähler manifolds and Gromov's notion of Kähler hyperbolicity. The second one is defined by weakening the boundedness condition of the $d$-potential of $\widetilde{\omega}$ to some $L^p$-condition, for $1\leq p\leq +\infty$, we shall call a manifold that admits such a metric \textbf{$L^p$-Kähler hyperbolic}. In a similar fashion, we will introduce another variant of SKT hyperbolicity, called \textbf{$L^p$-SKT hyperbolicity}.

We collect our deformation openness results in the following statement :
\begin{mthm}
Let $\sigma:\mathscr{X}\longrightarrow B$ be a holomorphic family of compact complex manifolds, where $B$ is an open ball around the origin in some $\C^N$. Let $X_t=\sigma^{-1}(t)$, for $t\in B$. 
\begin{enumerate}
    \item If $X_0$ is $L^p$-Kähler hyperbolic for some $1\leq p\leq+\infty$, then the fibers $X_t$ are $L^p$-SKT hyperbolic for all $t\in B$ sufficiently close to $0$.
    \item If $X_0$ is Kähler hyperbolic, then these fibers $X_t$ are $L^2$-Kähler hyperbolic. If moreover the hodge number $h^{2,0}(X_0)=0$, then the Kähler hyperbolicity property is \textbf{open} under small deformations.
    \item The Strong Kähler hyperbolicity property is \textbf{open} under deformations.
\end{enumerate}
\end{mthm}
\large{\textbf{Acknowledgments.}} This work is part of the author’s PhD thesis, and he would like to express his gratitude to his supervisor Dan Popovici for his constant guidance and support while carrying out this work. The author is also very grateful to Philippe Eyssidieux and Abdelghani Zeghib for their availability and willingness to answer his questions providing valuable insights throughout this work and to his co-office Louis Dailly for the useful exchanges they had about Kähler hyperbolicity and related topics. Many thanks are also due to Anna Fino for her feedback on the results of this paper and to Benoît Claudon for pointing out some interesting remarks on an earlier version of the manuscript and for sharing some insightful references.
\section{Some hyperbolicity notions on complex manifolds}
Throughout this paper, $X$ will be a compact connected complex manifold of dimension $\dim_\C X=n$ unless otherwise stated, and $\pi:\widetilde{X}\longrightarrow X$ its universal cover. In particular, if $\omega$ is some Hermitian metric on $X$, then $\widetilde{\omega}=\pi^*\omega$ is a \textit{complete} Hermitian metric on $\widetilde{X}$ since $\pi$ respects bidegrees by holomorphicity, and the pullback of a complete metric is also complete. Let us recall some basic definitions :

\begin{dfn} Let $X$ be a compact Kähler manifold. Then :
\begin{enumerate}
    \item\cite{li2019kahler} A Kähler metric $\omega$ on $X$ is said to be \textbf{Kähler $\widetilde{d}$-exact} if $\widetilde{\omega}$ is exact, i.e. there exists a $($globally defined$)$ smooth 1-form $\eta$ on $\widetilde{X}$ such that $\widetilde{\omega}=d\eta$. If such a metric exists on $X$, we say that $X$ is a \textbf{Kähler $\widetilde{d}$-exact} manifold.
    \item\cite{gromov1991kahler} If moreover the $d$-potential $\eta$ in $(1)$ can be chosen to be $\widetilde{\omega}$-bounded $($and not necessarily smooth$)$, then $\omega$ is said to be $\widetilde{d}$-\textbf{bounded}.  If such a metric exists on $X$, we say that $X$ is \textbf{Kähler hyperbolic}.
\end{enumerate}  
\end{dfn}
\begin{rmk}
\begin{enumerate}
    \item One can drop the smoothness assumption on the $d$-potential of $\widetilde{\omega}$ in the Kähler $\widetilde{d}$-exactness defintion. Indeed, suppose that $X$ admits a Kähler metric $\omega$ such that $\widetilde{\omega}=d\eta$, where $\eta$ is a 1-from which is not necessarily smooth on $\widetilde{X}$. This means that $[\widetilde{\omega}]=0\in H^2_{cur}(\widetilde{X},\R)$, where :
    \begin{equation*}
        H^2_{cur}(\widetilde{X},\R)=\dfrac{\{\text{real valued closed 2-currents on $X$}\}}{\{\text{real valued exact 2-currents on $X$}\}}
    \end{equation*}
    But since for every degree $k$ we have (Théorème 14, Chapitre IV in \cite{de1955varietes}) : 
    \begin{equation*}
        H^k_{cur}(\widetilde{X},\R)\simeq H^k_{DR}(\widetilde{X},\R)
    \end{equation*}
    we get : $[\widetilde{\omega}]=0\in H^2_{DR}(\widetilde{X},\R)$, which means that $\omega$ is Kähler $\widetilde{d}$-exact and $\widetilde{\omega}$ admits a smooth $d$-potential. Moreover, we can recover a smooth bounded $d$-potential of $\widetilde{\omega}$ in the Kähler hyperbolicity definition using the \textbf{heat equation method} (See the end of the proof of Theorem 2.1 in \cite{chen2021euler}, p.11-p.12). 
    \item Clearly Kähler hyperbolicity implies Kähler $\widetilde{d}$-exactness. Note that these two notions depend only on the cohomology class of the Kähler form $\omega$, in the following sense : if one representative of $\{\omega\}\in H_{DR}^2(X)$ is $\widetilde{d}$-exact (resp. $\widetilde{d}$-bounded), then all representatives are $\widetilde{d}$-exact (resp. $\widetilde{d}$-bounded) $($\cite{chen2018compact}, Lemma 2.4$)$. 
\end{enumerate} \label{26}
\end{rmk}
This leads to the following definition :
\begin{dfn}\cite{brunnbauer2024atoroidal} A cohomology class $\{\alpha\}\in H^k_{DR}(X,\R)$ is \textbf{hyperbolic} if $\alpha$ is \\ $\widetilde{d}$-bounded. We denote by $V^k_{hyp}(X)\subset H^k_{DR}(X,\R)$ the space of all hyperbolic classes. 
\end{dfn}
\begin{rmk}
    In the case $k=2$, the subspace $V^2_{hyp}(X)\subset H^2_{DR}(X,\R)$ depends only on the fundamental group of $X$ (see Theorem 5.1 in \cite{kedra2009symplectically}, Theorem 2.4 in \cite{brunnbauer2024atoroidal} or Proposition A.3 in \cite{bei2025geometric} for an alternative proof). This fact has been pointed out to the author by B. Claudon.  \label{71}
\end{rmk}
Let $\mathcal{K}_X\subset H^{1,1}(X,\R):=H^{1,1}_{BC}(X,\C)\cap H^2_{DR}(X,\R)$ be the \textit{Kähler cone} of $X$, namely, the set of all Kähler classes on $X$. Then :
\begin{enumerate}
    \item $X$ is Kähler $\widetilde{d}$-exact $\Longleftrightarrow$ $\left\{\{\omega\}\in\mathcal{K}_X\hspace{2pt}:\hspace{2pt}\widetilde{\omega}\in d\left(\mathscr{C}^\infty\left(\widetilde{X},T^*_{\widetilde{X}}\right)\right)\right\}\neq\emptyset$.
    \item $X$ is Kähler hyperbolic $\Longleftrightarrow$ $\mathcal{K}_X\cap V^2_{hyp}(X)\neq\emptyset$.
\end{enumerate} 
The basic examples of Kähler hyperbolic manifolds are Kähler manifolds homotopic to negatively curved Riemannian manifolds \cite{chen2018compact} and submanifolds of Hermitian locally symmetric manifolds of non-compact manifolds. Kähler $\widetilde{d}$-exact manifolds include other examples, such as Kähler non-elliptic manifolds; these are Kähler manifolds that admit Kähler metrics $\omega$ such that $\widetilde{\omega}=d\eta$ where $\eta$ is a $1$-form with \textbf{sublinear growth}, i.e.
\begin{equation}
    \lvert\eta\rvert_{\widetilde{\omega}}(x)\leq C(1+\hspace{1pt}d_{\widetilde{\omega}}(x,x_0))\hspace{4pt},\qquad\quad\forall x\in\widetilde{X}
    \label{16}
\end{equation}
for some constant $C>0$, where $d_{\widetilde{\omega}}$ is the distance induced by $\widetilde{\omega}$ and $x_0\in\widetilde{X}$ is some fixed base point (see \cite{jost1998vanishing}).
\begin{dfn}
Let $X$ be a complex manifold (not necessarily compact), then it is said to be \textbf{Kobayashi hyperbolic} if the "Kobayashi pseudo-distance" $($as defined in \cite{kobayashi1967intrinsic}$)$ is a genuine distance on $X$.
\end{dfn}
R. Brody  gave a characterization of Kobayashi in the compact setting \cite{brody1978compact}, answering a couple of questions raised by S. Kobayashi in \cite{kobayashi1974some} and \cite{kobayashi2005hyperbolic}. Namely :
\begin{thm} Let $X$ be a compact complex manifold. Then :
\begin{center}
    X is Kobayashi hyperbolic $\Longleftrightarrow$ Every entire curve in $X$ is constant
\end{center}
\end{thm}
The latter property has come to be called \textbf{Brody hyperbolicity}.

Recently, generalized notions of hyperbolicity were introduced in the non-Kähler context. Here, we recall the definition of some special metrics and their associated hyperbolicity notions. 
\begin{dfn} Let $X$ be a compact complex manifold of dimension $\dim_\C X=n$. \\ A Hermitian metric $\omega$ on $X$ is called :
\begin{enumerate} 
    \item \textbf{Strong Kähler with torsion} $($\textbf{SKT}$)$ or \textbf{pluriclosed} if $\partial\bar\partial\omega=0$.
    \item \cite{gauduchon1977theoreme} \textbf{Gauduchon} if $\partial\bar\partial\omega^{n-1}=0$
    \item $($\cite{gauduchon1977fibres},\cite{michelsohn1982existence}$)$ \textbf{Semi-Kähler} or \textbf{Balanced} if $d\omega^{n-1}=0$.
    \item \cite{popovici2013deformation} \textbf{Strongly Gauduchon} $($\textbf{sG}$)$ if $\partial\omega^{n-1}$ is $\bar\partial$-exact.
\end{enumerate}
\end{dfn}
\begin{dfn} Let $X$ be a compact complex manifold of dimension $\dim_\C X=n$. We say that $X$ is :
\begin{enumerate}
    \item $($\cite{marouani2023balanced},\cite{marouani2022some}$)$ \textbf{Balanced hyperbolic} if there exists a balanced metric $\omega$ on $X$ such that $\omega^{n-1}$ is $\widetilde{d}$-bounded.
    \item $($\cite{marouani2023skt}$)$ \textbf{SKT hyperbolic} if it admits an SKT metric $\omega$ which is \textbf{$(\widetilde{\partial+\bar\partial})$-bounded}, i.e. $\widetilde{\omega}=\partial\alpha+\bar\partial\beta$, where $\alpha$ and $\beta$ are bounded $(0,1)$ and respectively $(1,0)$-forms on $\widetilde{X}$. 
    \item $($\cite{marouani2023skt} \textbf{Gauduchon hyperbolic} if there exists a Gauduchon metric $\omega$ on $X$ such that $\widetilde{\omega}^{n-1}=\partial\alpha+\bar\partial\beta$, where $\alpha$ and $\beta$ are bounded $(n-2,n-1)$ and respectively $(n-1,n-2)$-forms on $\widetilde{X}$; that is to say that $\omega^{n-1}$ is $(\widetilde{\partial+\bar\partial})$-bounded.
    \item $($\cite{ma2024strongly}$)$ \textbf{sG hyperbolic} if there exists an sG metric $\omega$ on $X$ such that $\omega^{n-1}$ which is the $(n-1,n-1)$-component of a real $d$-closed $(2n-2)$-form $\Omega$ with $\pi^*\Omega=\eta$, where $\eta$ is a bounded $(2n-3)$-form on $\widetilde{X}$.
\end{enumerate}
\end{dfn}

M. Gromov mentioned in his paper \cite{gromov1991kahler} that Kähler hyperbolic manifolds are Kobayashi hyperbolic. The same result holds for SKT hyperbolic manifolds (\cite{marouani2023skt}, Theorem 3.5).

We will now recall some of the basic properties of Kähler $\widetilde{d}$-exact and Kähler hyperbolic manifolds.
\begin{prop}
    Let $X$ be a Kähler $\widetilde{d}$-exact manifold. Then there is no compact complex analytic subset $Z$ of $\widetilde{X}$ of dimension $k=\dim_\C Z>0$. In particular, $X$ admits no simply connected complex submanifold $Y$ $($such as rational curves$)$. \label{7}
\end{prop}
\begin{proof}
    Suppose that such a $Z$ exists, then :
    \begin{equation}
        \text{Vol}_{\widetilde{\omega}}(Z)=\dfrac{1}{k!}\int_Z\widetilde{\omega}^k=\dfrac{1}{k!}\int_Z d(\eta\wedge\widetilde{\omega}^{k-1})\overset{(*)}{=}0
    \end{equation}
    where $(*)$ is due to the compactness of $Z$. This is a contradiction since $\text{Vol}_{\widetilde{\omega}}(Z)>0$ by the positivity of the dimension. Now let $f:Y\hookrightarrow X$ be a simply connected submanifold of $X$, then $f$ lifts to a map $\widetilde{f}:Y\hookrightarrow\widetilde{X}$. So similarly, one gets $f^*\omega=\widetilde{f}^*\widetilde{\omega}^*=0$, hence $f$ is constant, which is again a contradiction.
\end{proof}
This implies by a result of S. Mori in \cite{mori1979projective}, that a \textit{projective} Kähler $\widetilde{d}$-exact manifold has a nef canonical bundle, due to the absence of rational curves. A rather interesting question was asked by Gromov about the ampleness of the canonical bundle $K_X$ of a Kähler hyperbolic manifold $X$. This question has been answered positively by B.-L. Chen and X. Yang :
\begin{prop}[\cite{chen2018compact}, Theorem 2.11.] If $X$ is a compact Kähler hyperbolic manifold, then its canonical bundle $K_X$ is ample.
\end{prop}
This has been done through combining the fact that $K_X$ is big (\cite{gromov1991kahler}, Corollary 0.4C) and cone theorems from \cite{mori1982threefolds} and \cite{kawamata1991length}.
Hence, this gave a partial answer to a conjecture proposed by S. Kobayashi :
\begin{conj}$($\textbf{Kobayashi '70}$)$ Let X be a compact Kähler $($or projective$)$ manifold which is Kobayashi hyperbolic. Then $K_X$ is ample. \label{63}
\end{conj}
A result that can be deduced from Proposition \ref{7} is the fact that Kähler $\widetilde{d}$-exact manifolds have \textit{generically large} fundamental group in the sense of Kollár (Definition 4.6 in \cite{kollar1995shafarevich}). Namely, we have :
\begin{cor}
    Let $X$ be a Kähler $\widetilde{d}$-exact manifold and $\iota:Z\hookrightarrow X$ an irreducible compact analytic subvariety of dimension $k=\dim_\C Z>0$. Then the image of $\iota_*:\pi_1(Z)\longrightarrow\pi_1(X)$ is infinite. In particular, $\pi_1(X)$ is infinite. \label{13}
\end{cor}
\begin{proof}
    If the image of  $\iota_*:\pi_1(Z)\longrightarrow\pi_1(X)$ were finite, passing to a finite cover $\nu:\hat{Z}\longrightarrow Z$ yields a map $\iota\circ\nu:\hat{Z}\longrightarrow X$ that lifts to the universal cover $\widetilde{X}$ of $X$, hence giving rise to a compact subvariety $\widetilde{Z}$ of $\widetilde{X}$ such that $\pi_{\lvert\widetilde{Z}}:\widetilde{Z}\longrightarrow Z$ is finite. This implies that $\int_{\widetilde{Z}}(\pi^*\omega)>0$, which is a contradiction according to Proposition \ref{7}.
\end{proof}
In particular, this means that the fundamental group of a Kähler hyperbolic manifold is infinite. This gives a partially positive answer to a question raised by S. Kobayashi in \cite{kobayashi1974some}, asking whether a compact (Kobayashi) hyperbolic complex manifold always has an infinite fundamental group, which turned out later to be false in general, since there are plenty of simply connected examples, mainly given by smooth projective general complete intersections of high
degree.

A consequence that can be drawn from Corollary \ref{13} is that if $X$ is Kähler $\widetilde{d}$-exact, then the volume form on $\widetilde{X}$ admits a $d$-potential $\alpha$ with sublinear growth in the sense of (\ref{16}) (see \cite{sikorav2001growth} for the proof). If $X$ is Kähler hyperbolic, we have a little bit more.
\begin{dfn}\cite{sullivan1976cycles} A complete Riemannian manifold $(M^m,g)$ is said to be \textbf{not closed} $($or \textbf{open}$)$ \textbf{at infinity} if it satisfies one of the following equivalent properties :
\begin{enumerate}
    \item $M^m$ admits a \textbf{bounded} volume form $\Omega>0$ $(C^{-1}dV_g\leq\Omega\leq CdV_g$ for some constant $C>0)$  with bounded $d$-potential, i.e. $\Omega=d\Gamma$ such that $\Gamma$ is a bounded $(m-1)$-form with respect to the Riemannian metric $g$.
    \item The following \textbf{linear isoperimetric inequality} holds true : There exists a constant $C>0$ such that for any bounded domain $D\subset M^m$ with $\mathscr{C}^1$-boundary $\partial D$ we have : 
\begin{equation}
    \text{Vol}_g(D)\leq C\hspace{1pt}\text{Area}_g(\partial D) \label{22}
\end{equation} 
\end{enumerate}
\end{dfn}
The proof of the equivalence can be found in \cite{sullivan1976cycles} or \cite{gromov1981hyperbolic}. It is easy to see that the universal cover $\widetilde{X}$ of a Kähler hyperbolic manifold ($X,\omega$) is open at infinity, since :
\begin{equation}
    \widetilde{\omega}^n=d(\eta\wedge\widetilde{\omega}^{n-1})
\end{equation}
and we have :
\begin{equation}
    \lVert\eta\wedge\widetilde{\omega}^{n-1}\rVert_{L^\infty_{\widetilde{\omega}}(\widetilde{X})}\leq\lVert\eta\rVert_{L^\infty_{\widetilde{\omega}}(\widetilde{X})}\lVert\omega\rVert^{n-1}_{L^\infty_{\omega}(X)}<+\infty
\end{equation}
One of the main purposes behind defining this notion of Kähler hyperbolicity was to give a partial affirmative answer to the Hopf-Chern conjecture :
\begin{conj} $($\textbf{Hopf-Chern}, \cite{yau1982problem}, Problem 10.$)$ Let $M$ be a compact Riemannian manifold of real dimension $2n$ of negative sectional curvature. Then :
\begin{equation}
    (-1)^n\chi(M)>0
\end{equation}
where $\chi(M)$ is the topological Euler characteristic of $M$. 
\end{conj}
This has been answered affirmatively in the case of Kähler hyperbolic manifolds (which includes compact Kähler manifolds with negative sectional curvature, see \cite{gromov1991kahler}). The main tool to prove this was a lower bound on the spectrum of the Hodge-de Rham Laplace operator $\Delta=dd^*+d^*d$ acting on $L^2$-forms of degree $k\neq n$. Namely, we have the following spectral gap result :
\begin{thm}\cite{gromov1991kahler} Let $X^n$ be a complete Kähler manifold equipped with a $d$-bounded Kähler form $\omega$, i.e. $\omega=d\eta$, where $\eta$ is a bounded $1$-form. Then, there exists a constant $\lambda_0=\lambda_0(\lVert\eta\rVert_{L^\infty_\omega(X)})>0$ such that for every $L^2$-form $\psi$ of degree $k\neq n$ on $X$, we have :
\begin{equation}
    \langle\Delta\psi,\psi\rangle\geq\lambda_0^2\lVert\psi\rVert^2_{L^2_\omega(X)} \label{24}
\end{equation}
Furthermore, the inequality (\ref{24}) still holds for $L^2$-forms of degree $n$ which are orthogonal to the harmonic $n$-forms.
\label{33}
\end{thm} 
In particular, this holds true on the universal cover of a compact Kähler hyperbolic manifold. In a similar fashion,  S. Marouani proved this spectral gap when $X$ is a complete Kähler manifold endowed with a ($\partial+\bar\partial$)-bounded Kähler form, i.e. $\omega=\partial\alpha+\bar\partial\beta$, where $\alpha$ and $\beta$ are bounded (see \cite{marouani2023skt}, Theorem 3.19).
\section{New hyperbolicity notions on compact Kähler manifolds}
\subsection{\texorpdfstring{$L^p$}{Lp}-Kähler and \texorpdfstring{$L^p$}{Lp}-SKT hyperbolicity}
Now, we want to define a modified version of hyperbolicity on Kähler manifolds by relaxing the boundedness condition on the $d$-potential of the pullback of the Kähler form $\omega$ to the universal cover $\widetilde{X}$. We will also set a similar definition for the SKT case since it will be important for later purposes, specially in Section 4.
\begin{dfn}
    Let $X$ be a compact Riemannian manifold and let $\pi:\widetilde{X}\longrightarrow X$ be its universal cover. A \textbf{fundamental domain} of $\widetilde{X}$ is an open relatively compact subset $D\Subset\widetilde{X}$ that satisfies the following conditions : 
    \renewcommand{\labelenumi}{(\textit{\roman{enumi}})}
    \begin{enumerate}
        \item The collection $\{\gamma\overline{D}\}_{\gamma\in\pi_1(X)}$ covers $\widetilde{X}$, i.e. $\underset{\gamma\in\pi_1(X)}{\bigcup}\gamma\overline{D}=\widetilde{X}$.
        \item $\partial D:=\overline{D}\setminus D$ has zero measure and $\gamma D\cap D=\emptyset$, for all $\gamma\in\pi_1(X)\setminus\{Id\}$.
    \end{enumerate}
\end{dfn}
\begin{rmk} \begin{enumerate}
    \item (\cite{ma2007holomorphic}-Section 3.6.1, \cite{atiyah1976elliptic}-Section 3) A way to construct this type of domains is the following : Let $\{U_i\}_{1\leq i\leq N}$ be a finite cover of $X$ by small balls, so that we have a continuous section $s_i:\widetilde{X}\longrightarrow X$ over $U_i$, and put $V_i=U_i\setminus\left(\underset{j<i}{\bigcup}\bar U_i\cap U_i\right)$. Then one can check that $D=\underset{1\leq i\leq N}{\bigcup}s_i(V_i)$ is a fundamental domain of $\widetilde{X}$.
    \item We have : $\pi(\overline{D})=X$ and $\pi_{\lvert D}:D\longrightarrow X$ is a diffeomorphism onto its image. More precisely, we have the following diagram :
    \begin{align*}
\xymatrix{ \overline{D}\hspace{2pt}
\ar@{^(->}[r]^{\iota}  \ar[d]_{\pi_{\overline{D}}} & \widetilde{X} 
\ar[d]^{\pi}\\
F \ar[r]_{\phi}^\simeq   & X  }
\end{align*}
    where $F=\overline{D}/\pi_1(X)$, $\iota$ is the inclusion map and $\phi$ is a diffeomorphism (see the proof of Theorem 9.2.4 in \cite{beardon2012geometry}).
    \item It can be shown that for any Riemannian metric $g$ on $X$ and for any two fundamental domains $D$ and $D'$ of $\widetilde{X}$, one has : $\text{Vol}_{\widetilde{g}}(D)=\text{Vol}_{\widetilde{g}}(D')$ (\cite{katok1992fuchsian}, Theorem 3.1.1). In particular, we have : $\text{Vol}_{\widetilde{g}}(\gamma D)\overset{(a)}{=}\text{Vol}_{\widetilde{g}}(D)\overset{(b)}{=}\text{Vol}_g(X)$, for any fundamental domain $D$ of $\widetilde{X}$ and for any $\gamma\in \pi_1(X)$, where (a) comes from the fact $\pi_1(X)$ acts by isometries on $\widetilde{X}$ (Actually, one can easily check that if $D$ is a fundamental domain of $\widetilde{X}$, then so is $\gamma D$, for any $\gamma\in\pi_1(X)$). As for (b), see \cite{katok1992fuchsian}, p.75.
\end{enumerate} \label{1}
\end{rmk}
\begin{dfn} Let $(X,g)$ be a compact Riemannian manifold, and let $D$ be a fundamental domain of $\widetilde{X}$. Let $\alpha$ be a $k$-form on $\widetilde{X}$, and let $1\leq p\leq+\infty$. We say that $\alpha$ satisfies the \textbf{property $($\textbf{F}$_p)$} if :
\begin{equation*}
   \textbf{$(F_p)$ : }\quad \exists\hspace{1pt}C>0\hspace{1pt},\quad\underset{\gamma\in\pi_1(X)}{\sup}\lVert\alpha\rVert_{L^p_{\widetilde{g}}(\gamma D)}\leq C
\end{equation*}
\end{dfn} 
\begin{rmk}
    \begin{enumerate}
    \item The property ($F_p$) does not depend on the choice of the Riemannian metric $g$ on $X$, since for any two Riemannian metrics $g$ and $g'$, the compactness of $X$ implies the existence of a constant $A>0$ such that : 
\begin{equation*}
    \dfrac{1}{A}g'\leq g\leq A g' \Longrightarrow  \dfrac{1}{A}\widetilde{g}'\leq\widetilde{g}\leq A\widetilde{g}'
\end{equation*}
    \item The property ($F_p$) does not depend also on the choice of the fundamental domain. Namely, if $\alpha$ satisfies the property ($F_p$) relatively to some fundamental domain $D$ of $\widetilde{X}$, then it satisies ($F_p$) for every other fundamental domain of $\widetilde{X}$. Indeed, let $\mu_g$ be the Riemannian volume measure induced by the metric $g$, and let $D'$ be another fundamental domain of $\widetilde{X}$, then there exists a finite number of sets $\{\gamma D\}_{\gamma\in\pi_1(X)}$ (let us say $\{\gamma_1 D,\dots,\gamma_k D\}$) such that $D'\underset{\text{a.e}}{\subset}\underset{1\leq j\leq k}{\bigcup}\gamma_j D$ in the sense that : 
    \begin{equation}
        \mu_g\left(\underset{1\leq j\leq k}{\bigcup}\gamma_j D\setminus D'\right)=0 \label{17}
    \end{equation}
    Hence, for any $\gamma_X\in\pi_1(X)$ we have : 
    \begin{equation}
        \int_{\gamma_X D'}\lvert\alpha\rvert^p_{\widetilde{g}}\mu_g\leq\int_{\gamma_X\left(\underset{1\leq j\leq k}{\bigcup}\gamma_j D\right)}\lvert\alpha\rvert^p_{\widetilde{g}}\mu_g\leq k\underset{1\leq j\leq k}{\sup}\int_{\gamma_{X,j}D}\lvert\alpha\rvert^p_{\widetilde{g}}\mu_g \label{3}
    \end{equation}
    where $\gamma_{X,j}=\gamma_X\cdot\gamma_j$. Thus (\ref{3}) implies that
    \begin{equation}
        \underset{\gamma\in\pi_1(X)}{\sup}\lVert\alpha\rVert_{L^p_{\widetilde{g}}(\gamma D')}\leq C\hspace{1pt}k^{\frac{1}{p}}=:C'>0
    \end{equation}
    Which gives us the property $(F_p)$ for $\alpha$ on $D'$ as well.
    \end{enumerate} \label{19}
\end{rmk}
\begin{rmk}
    Since the fundamental group of a compact manifold is countable (this is due to the fact that every compact topological manifold is homotopy equivalent to a finite CW-complex \cite{west1977mapping}, the latter has a finitely presented fundamental group), and a countable union of sets of measure zero has measure zero, together with the property ($i$) of fundamental domains, leads to the following : If $D$ is a fundamental domain of $\widetilde{X}$, then : 
    \begin{equation}
        \underset{\gamma\in\pi_1(X)}{\sup}\lVert\alpha\rVert_{L^p_{\widetilde{g}}(\gamma D)}=0 \Longrightarrow \alpha=0\quad \text{almost everywhere} 
    \end{equation}  \label{2} \vspace{-15pt}
\end{rmk}
Let us now define the notion of $L^p$-Kähler hyperbolicity.
\begin{dfn} Let $1\leq p\leq+\infty$, we say that a compact complex manifold $X$ is :
\begin{enumerate}
    \item \textbf{\textit{L}$^p$-Kähler hyperbolic} if it admits a Kähler metric $\omega$ such that  $\widetilde{\omega}=d\eta$, where $\eta$ is a $1$-form $($not necessarily smooth$)$ satisfying the property $(F_p)$ on $\widetilde{X}$. Such a metric $\omega$ is called \textbf{\textit{L}$^p$-Kähler hyperbolic}.
    \item \textbf{\textit{L}$^p$-SKT hyperbolic} if it admits an SKT metric $\omega$ such that  $\widetilde{\omega}=\partial\alpha+\bar\partial\beta$, where $\alpha$ and $\beta$ are $(0,1)$ and respectively $(1,0)$-forms $($not necessarily smooth$)$ that satisfy the property $(F_p)$ on $\widetilde{X}$. Such a metric $\omega$ is called \textbf{\textit{L}$^p$-SKT hyperbolic}.
\end{enumerate}
\end{dfn}
\begin{rmk}
\begin{enumerate}
    \item Note that (when $p=+\infty$), $L^\infty$-Kähler (resp. $L^\infty$-SKT) hyperbolicity is equivalent to the classical notion of Kähler (resp. SKT) hyperbolicity for the reasons we mentioned in Remark \ref{2}.
    \item If a manifold $X$ is $L^p$-Kähler hyperbolic for some $1\leq p\leq +\infty$, then $X$ is $\widetilde{d}$-exact thanks to Remark \ref{26}. 
\end{enumerate}   \label{49}
\end{rmk} 
Now we will try to investigate the relation between Kähler (resp. SKT) hyperbolicity and $L^p$-Kähler (resp. $L^p$-SKT) hyperbolicity for $1\leq p<+\infty$.
\begin{lem}
    Let $\omega$ be a Hermitian metric on $X$ and let $\eta$ be a $\widetilde{\omega}$-bounded $k$-form on $\widetilde{X}$. Then $\eta$ has the property $(F_p)$ for every $1\leq p\leq +\infty$. \label{4}
\end{lem}
\begin{proof}
    The case $p=+\infty$ is due to Remark \ref{49}. Assume $p<+\infty$, and let $D$ be a fundamental domain of $\widetilde{X}$ and $\gamma\in\pi_1(X)$. Then :
    \begin{equation}
        \lVert\eta\rVert_{L^p_{\widetilde{\omega}}(\gamma D)}=\left(\int_{\gamma D}\lvert\eta\rvert_{\widetilde{\omega}}^p\widetilde{\omega}_n\right)^{\frac{1}{p}}\leq\lVert\eta\rVert_{L^\infty_{\widetilde{\omega}}(\widetilde{X})}(\text{Vol}_{\widetilde{\omega}}(\gamma D))^\frac{1}{p}\overset{(*)}{=}\lVert\eta\rVert_{L^\infty_{\widetilde{\omega}}(\widetilde{X})}(\text{Vol}_{\omega}(X))^\frac{1}{p}=:C
    \end{equation}
    where $C>0$ is independent of $\gamma$ and $(*)$ comes from the fact that $\pi_1(X)$ acts on $\widetilde{X}$ by isometries and Remark \ref{1}.(2). Hence $\eta$ has the property $(F_p)$.
\end{proof}
\begin{cor}
        For any $1\leq p\leq +\infty$, if $\alpha$ is a $d$-closed $k$-form on $X$ such that $\pi^*\alpha$ has a $d$-potential with the $(F_p)$ property, then for every $k$-form $\beta\in\{\alpha\}_{DR}\in H_{DR}^k(X,\C)$, $\pi^*\beta$ also has a $d$-potential with the $(F_p)$ property. \label{5}
\end{cor}
\begin{prop}
\begin{enumerate}
    \item If $X$ is a Kähler $($resp. SKT$)$ hyperbolic manifold, then $X$ is $L^p$-Kähler $($resp. $L^p$-SKT$)$ hyperbolic, for any $1\leq p\leq+\infty$. More generally, if $X$ is $L^p$-Kähler $($resp. $L^p$-SKT$)$ hyperbolic, then $X$ is $L^r$-Kähler $($resp. $L^r$-SKT$)$ hyperbolic for any $1\leq r\leq p$.
    \item Let $1\leq p\leq+\infty$. If $X$ and $Y$ are compact $L^p$-Kähler $($resp. $L^p$-SKT$)$ hyperbolic manifolds, then so is $X\times Y$.
\end{enumerate} \label{14}
\end{prop}
\begin{proof} We are going to do the proof for the Kähler case, and the SKT case can be treated similarly.
\begin{enumerate}
    \item If $X$ is Kähler hyperbolic, apply Lemma \ref{4}. to get the fact that $X$ is $L^p$-Kähler hyperbolic for any $1\leq p<+\infty$. Now let $\omega$ be an $L^p$-Kähler hyperbolic metric on $X$, and let $q>0$ be such that $\dfrac{1}{p}+\dfrac{1}{q}=\dfrac{1}{r}$, then $\dfrac{p}{r}$ and $\dfrac{q}{r}$ are conjugate Hölder indices. Hence, Hölder's inequality yields for any $\gamma\in\pi_1(X)$ :
    \begin{align}
        \int_{\gamma D}\lvert\eta\rvert_{\widetilde{\omega}}^r\widetilde{\omega}_n&\leq \left(\int_{\gamma D}\lvert\eta\rvert^p\widetilde{\omega}_n\right)^{\frac{r}{p}}\left(\int_{\gamma D}\widetilde{\omega}_n\right)^{\frac{r}{q}} \\ &=\lVert\eta\rVert^r_{L^p_{\widetilde{\omega}}(\gamma D)}\text{Vol}^{\frac{r}{q}}_\omega(X) \label{6}
    \end{align}
    where (\ref{6}) comes the Remark \ref{1}.3. Hence we have : 
    \begin{equation}
        \underset{\gamma\in\pi_1(X)}{\sup}\lVert\eta\rVert_{L^r_{\widetilde{\omega}}(\gamma D)}\leq\underset{\gamma\in\pi_1(X)}{\sup}\lVert\eta\rVert_{L^p_{\widetilde{\omega}}(\gamma D)}\text{Vol}_{\omega}^{\frac{1}{r}-\frac{1}{p}}(X)=:C'
    \end{equation}
    where $C'=C\text{Vol}_{\omega}^{\frac{1}{r}-\frac{1}{p}}(X)>0$ is obviously independent of $\gamma$.
    \item Let $\omega_X$ (resp. $\omega_Y$) be a Kähler metric on $X$ (resp. $Y$) such that $\widetilde{\omega}_X=d\eta_X$ (resp. $\widetilde{\omega}_Y=d\eta_Y$), where $\eta_X$ (resp. $\eta_Y$) have the ($F_p$) property with some constant $C_X>0$ (resp. $C_Y>0$). Let $pr_X$ (resp. $pr_Y$) be the projection of $X\times Y$ on $X$ (resp. on $Y$) and $\pi:\widetilde{X\times Y}\longrightarrow X\times Y$ be the universal cover of $X\times Y$. Then we have the following :
    \begin{enumerate}
        \item The projections $pr_X$ and $pr_Y$ lift to projections $\widetilde{pr_X}$ and $\widetilde{pr_Y}$ from $\widetilde{X\times Y}$ to $\widetilde{X}$ and respectively $\widetilde{Y}$ (since $\pi_1(X\times Y)\simeq\pi_1(X)\times\pi_1(Y)$ via the isomorphism given by : $[\gamma]\in\pi_1(X\times Y)\longmapsto([pr_X\circ\gamma],[pr_Y\circ\gamma])\in\pi_1(X)\times\pi_1(Y)$, and hence $\widetilde{X\times Y}\simeq\widetilde{X}\times\widetilde{Y}$ by unicity of the universal cover). In particular, we have :
        \begin{equation}
            pr_X\circ\pi=\pi_X\circ\widetilde{pr_X} \qquad\text{and}\qquad
            pr_Y\circ\pi=\pi_Y\circ\widetilde{pr_Y} \label{15}
        \end{equation}
        \item $\omega=pr_X^*\omega_X+pr_Y^*\omega_Y$ is a Kähler metric on $\widetilde{X}\times \widetilde{Y}$, and we have : 
        \begin{equation}
            \widetilde{\omega}=\pi^*pr_X^*\omega_X+\pi^*pr_Y^*\omega_Y\overset{(*)}{=}\widetilde{pr_X}^*\widetilde{\omega}_X+\widetilde{pr_Y}^*\widetilde{\omega}_Y=d(\widetilde{pr_X}^*\eta_X+\widetilde{pr_Y}^*\eta_Y) \label{18}
        \end{equation}
        where $(*)$ comes from the identities in (\ref{15}).
        \item If $D_X$ and $D_Y$ are fundamental domains of $\widetilde{X}$ and respectively $\widetilde{Y}$, then one can check that $D=D_X\times D_Y$ is a fundamental domain of $\widetilde{X}\times\widetilde{Y}$, and that for every $\gamma=(\gamma_X,\gamma_Y)\in\pi_1(X)\times\pi_1(Y)$ we have : $\gamma D=(\gamma_X D_X)\times(\gamma_Y D_Y)$. Hence, we get :
    \end{enumerate}
    \begin{align}
        \lVert\widetilde{pr_X}^*\eta_X+\widetilde{pr_Y}^*\eta_Y\rVert_{L^p_{\widetilde{\omega}}(\gamma D)}\leq&\lVert\widetilde{pr_X}^*\eta_X\rVert_{L^p_{\widetilde{\omega}}(\gamma D)}+\lVert\widetilde{pr_Y}^*\eta_Y\rVert_{L^p_{\widetilde{\omega}}(\gamma D)} \\
        &=\lVert\eta_X\rVert_{L^p_{\widetilde{\omega}_X}(\gamma_X D_X)}+\lVert\eta_Y\rVert_{L^p_{\widetilde{\omega}_Y}(\gamma_Y D_Y)}
    \end{align}
    Thus : $\underset{\gamma\in\pi_1(X\times Y)}{\sup}\lVert\widetilde{pr_X}^*\eta_X+\widetilde{pr_Y}^*\eta_Y\rVert_{L^p_{\widetilde{\omega}}(\gamma D)}\leq C_X+C_Y=:C>0$, which means that the $d$-potential of $\widetilde{\omega}$ given in (\ref{18}) satisfies the ($F_p$) property (which does not depend on the choice of the fundamental domain of $\widetilde{X}\times\widetilde{Y}$ by Remark \ref{19}). Hence $X\times Y$ is $L^p$-Kähler hyperbolic. 
    \end{enumerate}
\end{proof} 
Note that for compact Riemann surfaces, this notion of $L^p$-Kähler hyperbolicity is not a new hyperbolicity notion. We already know that Kähler hyperbolicty implies Kobayashi/Brody hyperbolicity, which in turns implies negative curvature in dimension 1, which give us back Kähler hyperbolicity. Recall that we proved that Kähler hyperbolicty implies $L^p$-Kähler hyperbolicity in Proposition \ref{14}. Now, we show the following :
\begin{prop}
    Let $X$ be a compact Riemann surface, and let $1\leq p\leq+\infty$. Then :
    \begin{center}
        $X$ is $L^p$-Kähler hyperbolic  $\Longrightarrow$ $X$ is Brody hyperbolic
    \end{center}
\end{prop}
By Proposition \ref{14}, it is sufficient to show this result for $p=1$. The proof uses the celebrated \textit{uniformization theorem} established at the beginning of the 20th century by H. Poincaré \cite{poincare1908uniformisation} and P. Koebe \cite{koebe1907uniformisierung} independently. It states that every simply connected Riemann surface is biholomorphic to the Riemann sphere $\CP^1$, the complex plane $\C$, or the unit disc $\D$.
\begin{proof} To show that $X$ is Brody hyperbolic, we shall argue that the only possibility for the universal cover $\widetilde{X}$ of $X$ is to be biholomorphic to $\D$.

First, we can easily rule out the case $\widetilde{X}\simeq\CP^1$, since a Kähler form on $\CP^1$ cannot be exact (because $\CP^1$ is compact).

It remains to show that $\widetilde{X}$ cannot be biholomorphic to $\C$, i.e. $X$ is not a complex torus. Assume that this is the case, then by the \textbf{Hodge isomorphism theorem}, one has : $H^{1,1}(X,\R)\simeq\mathcal{H}^{1,1}(X,\R)$. On the other hand, note that a (1,1)-form $\alpha$ is of top degree (since $\dim_\C X=1)$, so it is proportional to the standard Kähler metric $\omega_0$ on the torus $X$ (the one induced by the euclidean metric on $\C$), i.e. $\alpha=f\omega_0$ for some $f\in\mathscr{C}^\infty(X)$. We consider $\Delta''_0$ to be the $\bar\partial$-Laplacian induced by $\omega_0$, then :
\begin{equation*}
    \Delta_0''\alpha=\Delta_0''(f\omega_0)=(\Delta_0'' f)\omega_0
\end{equation*}
since the Lefschetz operator ($\cdot\longmapsto\omega_0\wedge\cdot$) commutes with the $\bar\partial$-Laplacian on Kähler manifolds. Hence $\alpha$ is harmonic if and only if $f$ is harmonic. But since $X$ is compact, $f$ must be constant by the maximum principle. In particular, (1,1)-classes on $X$ are in a one-to-one correspondence with constant hermitian forms on $\mathbb{C}$.

Thanks to Corollary \ref{5}, we only need to show that there is no $\Delta_0''$-harmonic Kähler metric $\omega$ such that $\widetilde{\omega}=d\eta$, where $\eta$ satisfies the ($F_1$) property. We may write :
\begin{equation}
    \widetilde{\omega}=c\dfrac{i}{2}dz\wedge d\bar z=cdx\wedge dy
\end{equation}
where $z=x+iy\in\C$, $x,y\in\R$. We can assume (without loss of generality) that $c=1$, and that the fundamental group $\pi_1(X)\simeq\Z^2$ is the standard lattice in $\C$ generated by $\gamma_1=(1,0)$ and $\gamma_2=(0,1)$. One can write $\widetilde{\omega}$ in polar coordinates $(r,\theta)$ :
\begin{equation}
    \widetilde{\omega}=d(r\cos\theta)\wedge d(r\sin\theta)=rdr\wedge d\theta=dr\wedge d\sigma
\end{equation}
where $d\sigma=rd\theta$ is a 1-form on $\C$ that induces a length measure on the circles $\Sp_r\subset\C$. Namely, we have :
\begin{equation}
    L_{\widetilde{\omega}}(\Sp_r):=\int_{\Sp_r}d\sigma_{\lvert\Sp_r}>0\hspace{4pt},\qquad\quad\forall r>0
\end{equation}
In particular, one can express $\eta$ on circles as follows :
\begin{equation}
    \eta_{\lvert\Sp_r}=f_r(\theta)d\sigma_{\lvert\Sp_r}\hspace{4pt},\qquad\forall r>0\hspace{2pt},\quad\forall 0\leq\theta<2\pi
\end{equation}
Simple computations yields that $\lvert d\sigma\rvert_{\widetilde{\omega}}(z)=1$, $\forall z\in\C$. In particular, $\lvert\eta_{\lvert\Sp_r}\rvert_{\widetilde{\omega}}=\lvert f_r(\theta)\rvert$, for all $r>0$. Now, we measure the volume growth of the disks in $\C$. Using Stokes' theorem, we get for every $r>0$ :
\begin{equation}
    A_{\widetilde{\omega}}(\D_r)=\int_{\D_r}\widetilde{\omega}=\int_{\D_r}d\eta=\int_{\Sp_r}\eta=\int_{\Sp_r}f_r(\theta)d\sigma_{\lvert\Sp_r}\leq\int_{\Sp_r}\lvert\eta\rvert_{\widetilde{\omega}}d\sigma_{\lvert\Sp_r}
\end{equation}
This implies that for every $R>0$ :
\begin{equation}
    F(R):=\int_0^R A_{\widetilde{\omega}}(\D_t)dt\leq\int_0^R\left(\int_{\Sp_t}\lvert\eta\rvert_{\widetilde{\omega}}d\sigma_{\lvert\Sp_t}\right)dt=\int_{\D_R}\lvert\eta\rvert_{\widetilde{\omega}}\widetilde{\omega}
\end{equation}
We denote by $C_{2r}:=\{(x,y)\in\R^2\hspace{1pt}:\hspace{1pt}\max\{\lvert x\rvert,\lvert y\rvert\}\leq r\}$ the square of side length $2r>0$ with center the origin, and let $D$ be the unit square passing through the origin (which is a fundamental domain of $\C$ for the prescribed action of $\pi_1(X)$). Then obviously $\D_r\subset C_{2r}$ for any $r>0$, hence we get :
\begin{equation}
    F(R)\leq\int_{\D_R}\lvert\eta\rvert_{\widetilde{\omega}}\widetilde{\omega}\leq\int_{C_{2R}}\lvert\eta\rvert_{\widetilde{\omega}}\widetilde{\omega}\leq N_R\underset{\gamma\in\Gamma_R}{\sup}\int_{D+\gamma}\lvert\eta\rvert_{\widetilde{\omega}}\widetilde{\omega}\hspace{4pt},\qquad\forall R>0 \label{65}
\end{equation}
where $\Gamma_R=\{\gamma\in\pi_1(X)\hspace{1pt}:\hspace{1pt}(D+\gamma)\cap C_{2R}\neq\emptyset\}$, and $N_R=\lvert\Gamma_R\rvert$. Then $N_R=4\lceil R\rceil^2$, $\forall R>0$, where $\lceil\cdot\rceil$ is the \textit{upper integer part} function. On the other hand $\underset{\gamma\in\Gamma_R}{\sup}\int_{D+\gamma}\lvert\eta\rvert_{\widetilde{\omega}}\widetilde{\omega}\leq C$, for some constant $C>0$, since $\eta$ satisies the ($F_1$) property. But, we know that $A_{\widetilde{\omega}}(\D_t)=\pi t^2$, for any $t\geq0$, thus (\ref{65}) becomes :
\begin{equation}
    F(R)=\int_0^R\pi t^2dt=\dfrac{\pi}{3}R^3\leq4C\lceil R\rceil^2\hspace{4pt},\qquad\qquad\forall R>0
\end{equation}
This is a contradiction. Hence, $\widetilde{X}\not\simeq\C$, so $\widetilde{X}$ is necessarily biholomorphic to $\D$, which means that $X$ is Brody hyperbolic.
\end{proof}
Similar arguments can be used to show that complex tori $\C^n/\Lambda$ of any dimension $n$, are not $L^p$-Kähler hyperbolic for any $1\leq p\leq+\infty$. However, it is not obvious for us whether $L^p$-Kähler hyperbolicity still implies Brody hyperbolicity in higher dimensions.

At this point, it is also difficult to find examples of $L^p$-Kähler hyperbolic manifolds for $p<+\infty$ other than the Kähler hyperbolic ones.

\begin{qst}
    Let $X$ be an $L^p$-Kähler hyperbolic manifold. Is $X$ of general type ? \label{8}
\end{qst}
A potential way to answer this question is to prove the existence of a nonzero $L^q$-section of the pullback of $K_X$ to the universal cover $\widetilde{X}$, for some $q>0$ (Corollary 13.10, \cite{kollar1995shafarevich}).

A positive answer to Question \ref{8} implies that $K_X$ is in fact ample, since a Kähler $\widetilde{d}$-exact manifold of general type has ample canonical bundle thanks to Theorem 2.8, \cite{li2019kahler}. If one can show that $L^p$-Kähler hyperbolic manifolds are indeed Kobayashi/Brody hyperbolic, then answering Question \ref{8} means having a larger class of compact Kähler Kobayashi hyperbolic manifolds satisfying the Kobayashi conjecture \ref{63}.

\subsection{Strong Kähler hyperbolicity}
In this part, we will define a stronger hyperbolicity notion on compact Kähler manifolds which, despite the lack of examples, is interesting from the deformation theoretic point of view. We will start by recalling some basic definitions from \cite{gromov1987hyperbolic}.

Let ($M,d$) be a metric space, and fix a reference point $x_0\in M$. One defines the following \textit{inner product} on $M$ (also called the \textbf{Gromov product}) :
\begin{equation}
    (x,y)_{x_0}:=\dfrac{1}{2}(d(x,x_0)+d(y,x_0)-d(x,y))\hspace{4pt},\qquad\quad\forall x,y\in M \label{32}
\end{equation}
\begin{dfn} We say that $(M,d)$ is \textbf{hyperbolic} if there exists a reference point $x_0\in M$, one can find $\delta\geq 0$ such that the following inequality holds for the inner product defined in (\ref{32}):
\begin{equation}
    (x,y)_{x_0}\geq\min\{(x,z)_{x_0},(y,z)_{x_0}\}-\delta\hspace{4pt},\qquad\quad\forall x,y,z\in M \label{35}
\end{equation} \label{37}
\end{dfn}
\begin{rmk} (Corollary 1.1.B, \cite{gromov1987hyperbolic}) If the inequality (\ref{35}) is satisfied for some $x_0\in M$ with constant $\delta$, then a similar inequality holds for any base point $x\in M$ with constant $2\delta$. Hence, the hyperbolicity of ($M,d$) does not depend on the choice of the base point.
\end{rmk}
Now, let $\Gamma$ be a finite-type group and let $S\subset\Gamma$ be a finite set generating $\Gamma$. We define the \textbf{length} of an element $\gamma\in\Gamma$ with respect to $S$ as :
\begin{equation}
    \lVert\gamma\rVert_S:=\min\{ k\in\N\hspace{1pt}:\hspace{1pt}\gamma=s_1^{\varepsilon_1}\cdots s_k^{\varepsilon_k}\hspace{1pt},\hspace{1pt}s_i\in S\hspace{1pt},\hspace{1pt}\varepsilon_i=\pm1\hspace{1pt},\hspace{1pt}\forall1\leq i\leq k\} 
\end{equation}
This gives a distance on $\Gamma$ defined by :
\begin{equation}
    d_S(\gamma_1,\gamma_2):=\lVert\gamma_2^{-1}\gamma_1\rVert_S\hspace{4pt},\qquad\quad\forall\gamma_1,\gamma_2\in\Gamma 
\end{equation}
\begin{rmk}
    If $S$ and $S'$ are two finite sets generating a finite-type group $\Gamma$, then the metric spaces ($\Gamma,d_S$) and ($\Gamma,d_{S'}$) are quasi-isometric (see \cite{ghys1989groupes}, p.2 for instance). \label{44}
\end{rmk}
\begin{dfn}
    A finite-type group $\Gamma$ is \textbf{hyperbolic} $($or \textbf{Gromov hyperbolic}$)$ if $(\Gamma,d_S)$ is hyperbolic as a metric space in the sense of Definition \ref{37} for some finite generating set $S$ of $\Gamma$. \label{46}
\end{dfn}
More generally, one can associate to a finite-type group $\Gamma$ with finite generating set $S$ a \textbf{Cayley Graph} $\mathcal{G}(\Gamma,S)$. This is the graph whose vertices are the elements of $\Gamma$, and where an edge connects $\gamma_1$ to $\gamma_2$ if $d_S(\gamma_1,\gamma_2)$. This graph can be endowed with a natural metric, invariant under $\Gamma$, in which each edge is isometric to an interval of length 1 in $\R$, and such that the distance between two points is the shortest length of a path connecting them. Then the metric spaces ($\Gamma,d_S$) and $\mathcal{G}(\Gamma,S)$ are quasi-isometric by definition. In particular :
\begin{center}
    ($\Gamma,d_S$) is hyperbolic $\Longleftrightarrow$ $\mathcal{G}(\Gamma,S)$ is hyperbolic
\end{center}
\begin{rmk}
    Definition \ref{46} does not depend on the choice of the finite generating set of $\Gamma$ by Remark \ref{44}, since hyperbolicity is a quasi-isometry invariant (see \cite{ghys1989groupes}, Corollaire p.13 or \cite{ghys1990groupes}, Théorème 12).
\end{rmk}

There are several definitions for hyperbolicity in the literature, that can be found in the previously mentioned sources. For instance, the inequality (\ref{35}) ensures that on a \textit{geodesic} metric space ($M,d$) (for example, the Cayley graph of finite type group), all the triangles are \textbf{$\delta$-thin} for some constant $\delta\geq0$. This means that for any points $x,y,z\in M$, and any choice of segments $[x,y]$, $[y,x]$ and $[z,x]$, we have the following property : Any point of $[x,y]$ is at most at a distance $\delta\geq 0$ from some point in $[y,z]\cup[z,x]$. The proof of the equivalence of definitions for hyperbolicity can be found in detail in \cite{alonso1991notes}, Chapter 2.

\begin{dfn}
    We say that a compact Kähler manifold $X$ is \textbf{strongly Kähler hyperbolic} $($\textbf{SKH} for short$)$ if it is Kähler $\widetilde{d}$-exact and $\pi_1(X)$ is hyperbolic.
\end{dfn}
\begin{rmk}
    It is not hard to see that the Kähler $\widetilde{d}$-exactness condition is satisfied by $X$ if $\pi_2(X)$ is torsion (Remark 1.8, \cite{chen2021euler}). Indeed, the universal covering map induces an isomorphism of homotopy groups $\pi_k(\widetilde{X})\simeq\pi_k(X)$, $\forall k\geq 2$ since every map $f:\Sp^k\longrightarrow X$ with $k\geq 2$ lift to a map $\widetilde{f}:\Sp^k\longrightarrow\widetilde{X}$ by simple connectedness of $\Sp^k$ (which is not the case if $k=1$). In particular, $\pi_2(\widetilde{X})$ is torsion. This implies that $H_2(\widetilde{X},\Z)$ is torsion by the Hurewicz theorem ($H_2(\widetilde{X},\Z)\simeq\pi_2(\widetilde{X})$ since $\widetilde{X}$ is simply connected), which means that $H_2(\widetilde{X},\R)=0$. The universal coefficient theorem implies that $H^2(\widetilde{X},\R)=0$. By de Rham's theorem, we get that $H^2_{DR}(\widetilde{X},\R)=0$. This means that every closed $2$-form on $\widetilde{X}$ is exact. In particular, any Kähler metric $\omega$ on $X$ is $\widetilde{d}$-exact.
\end{rmk}

On the other hand, B-L. Chen and X. Yang proved that SKH manifolds are Kähler hyperbolic (Corollary 1.7). In particualr, SKH hyperbolicity is pinched between \textit{real hyperbolicity} (i.e. $\pi_1(X)$ is hyperbolic and $\pi_2(X)=0$) for compact Kähler manifolds and Kähler hyperbolicity. Namely, we have :
\begin{prop} Let $X$ be a compact Kähler manifold. Then :
    \begin{center}
  $X$ is real hyperbolic $\Longrightarrow$ $X$ is SKH $\Longrightarrow$ $X$ is Kähler hyperbolic  
\end{center} \label{11}
\end{prop} 

The following examples were kindly communicated to the author by P. Eyssidieux.
\begin{ex}
    \begin{enumerate}
        \item The only known class of examples of Kähler groups (fundamental groups of compact Kähler manifolds) that are hyperbolic is the class of fundamental groups of negatively curved Kähler manifolds. These manifolds are known to be Kähler hyperbolic as mentioned above (hence Kähler $\widetilde{d}$-exact). In particular, these are SKH manifolds.
        \item In complex dimension 1, Kähler hyperbolicity and the SKH property are clearly the same. However, this does not hold in higher dimensions (that is the left implication in Proposition \ref{11} is strict). An example of a Kähler hyperbolic manifold which is not SKH is the following : Let $X=C_1\times C_2$ be the product of two smooth curves of genus at least 2. It is obvious that $X$ is Kähler hyperbolic since both $C_1$ and $C_2$ are Kähler hyperbolic. But $\pi_1(X)$ contains $\Z^2$, which means that for any $\delta\geq0$, there are triangles that are not \textbf{$\delta$-thin} in $\pi_1(X)$. Hence $\pi_1(X)$ is not Gromov hyperbolic and $X$ is not SKH.
    \end{enumerate}
\end{ex}

\section{Deformations openness results}

This section will be dedicated to the study of the deformations of some of the notions mentioned above. Let us first recall some fundamental results in deformation theory :
\begin{dfn}
    A \textbf{holomorphic family of deformations} of a compact complex manifold $X$ is a complex analytic family $(\mathscr{X},B,\sigma)$, where $\mathscr{X}$ and $B$ are compact complex manifolds and $\sigma:\mathscr{X}\longrightarrow B$ is a proper holomorphic submersion such that $X_0:=\sigma^{-1}(0)=X$, for a fixed base point $0\in B$ $(B$ is in general an open ball around the origin in some $\C^N)$. 
\end{dfn}
The study of deformation openness for Kähler $\widetilde{d}$-exact and $L^p$-Kähler hyperbolic manifolds is motivated by two facts :

The first fact is that the Kähler property is open under holomorphic deformations :
\begin{thm}$($\cite{kodaira1960deformations}, Theorem 15$)$ If a fiber $X_0$ in a holomorphic family of compact complex manifolds is Kähler, then all the nearby fibers $X_t$ are also Kähler. Moreover, any Kähler metric $\omega_0$ on $X_0$ can be deformed to a $\mathscr{C}^\infty$ family of
Kähler metrics $\omega_t$ on the nearby fibers $X_t$. \label{45}
\end{thm}
And the second fact is that deformation openness results have been proved for some complex hyperbolicity notions, namely : 
\begin{thm}$($\cite{brody1978compact}, Theorem 3.1$)$ The Kobayashi hyperbolicity property is open under holomorphic deformations.
\end{thm}
\begin{thm}$($\cite{ma2024strongly}, Theorem 2.4$)$ The strongly Gauduchon hyperbolicity property is open under holomorphic deformations. \label{10}
\end{thm}

Recall that by Ehresmann’s Theorem \cite{ehresmann1947espaces}, every holomorphic family of compact complex manifolds is locally $\mathscr{C}^\infty$ trivial, i.e. the neighboring fibers are diffeomorphic. In particular, for every degree $k$ we have : $H^k_{DR}(X_t,\C)\simeq H^k_{DR}(X_0,\C)$, for $t$ close to 0 (and for all $t\in B$ if $B$ is contractible). So it makes sense to define the following :
\begin{dfn}
    Let ($X,\omega)$ be a compact Kähler manifold. We say that a deformation $X_t=\sigma^{-1}(t)$ of $X$ is \textbf{polarized} by $\{\omega\}_{BC}$ if $\{\omega\}_{DR}\in\hspace{0.5pt}H_{DR}^2(X)$ defines a K{\"a}hler class $($hence remains of type $(1,1))$ for the complex structure of $X_t$.
\end{dfn}

A direct application of Remark \ref{26} and Corollary \ref{5} yields our first deformation openness result regarding \textbf{polarized deformations} :
\begin{prop} Let $X$ be a compact Kähler manifold and $(\mathscr{X},B,\sigma)$ be a holomorphic family of deformations of $X$.
\begin{enumerate}
    \item If $X$ is Kähler $\widetilde{d}$-exact, then so is every fiber $X_t$ polarized by the fixed Kähler $\widetilde{d}$-exact class on $X$.
    \item If $X$ is $L^p$-Kähler hyperbolic for some $1\leq p\leq+\infty$, then so is every fiber $X_t$ polarized by the fixed $L^p$-Kähler hyperbolic class on $X$.
\end{enumerate}
\end{prop}
\begin{rmk}
    If moreover $X$ has a Gromov hyperbolic fundamental group in the first part of the previous proposition, i.e. $X$ is SKH, then the polarized fibers $X_t$ are also SKH since $X_t\underset{\mathscr{C}^\infty}{\simeq}X$, in particular $\pi_1(X_t)\simeq\pi_1(X)$.
    \label{51}
\end{rmk}
\subsection{Deformation of \texorpdfstring{$L^p$}{Lp}-Kähler hyperbolic manifolds}
Now we are going to deal with this deformation openness problem in its full generality. But before that, let us recall one important characterization of Kähler metrics. The following material can be found in \cite{kodaira1960deformations} or \cite{schweitzer2007autour} (see also \cite{popovici2015aeppli} for some proofs).
\begin{dfn} We define the $4^{th}$ order \textbf{Bott-Chern Laplacian} as follows :
\begin{equation*}
    \Delta_{BC}=\partial^*\partial+\bar\partial^*\bar\partial+(\partial\bar\partial)^*(\partial\bar\partial)+(\partial\bar\partial)(\partial\bar\partial)^*+(\partial^*\bar\partial)^*(\partial^*\bar\partial)+(\partial^*\bar\partial)(\partial^*\bar\partial)^*
\end{equation*}
\end{dfn}
This Laplacian is elliptic and formally self-adjoint, so it induces the following orthogonal decomposition on any compact Hermitian manifold $X$ for any $r,s\in\{0,\dots,n\}$ : 
\begin{equation}
    \mathscr{C}^{\infty}_{r,s}(X)=\mathcal{H}_{BC}^{r,s}(X)\oplus\partial\bar\partial(\mathscr{C}^{\infty}_{r-1,s-1}(X))\oplus(\partial^*(\mathscr{C}^{\infty}_{r+1,s}(X)+\bar\partial^*(\mathscr{C}^{\infty}_{r,s+1}(X))
\end{equation}
where $\mathcal{H}_{BC}^{r,s}(X)=\text{Ker}\Delta_{BC}\cap\mathscr{C}^\infty_{r,s}(X)$ has the following decomposition :
\begin{equation}
    \text{Ker}\Delta_{BC}=\text{Ker}\partial\cap\text{Ker}\bar\partial\cap\text{Ker}(\partial\bar\partial)^*=\text{Ker}\partial\cap\text{Ker}\bar\partial\cap\text{Ker}(\partial\bar\partial)^* \label{47}
\end{equation}
Hence we have the following :
\begin{lem}$($\cite{popovicinon}, Lemma 2.6.5$)$ Let $\omega$ be a Hermitian metric on a compact complex manifold $X$. Then :
\begin{equation}
    \omega\text{ is Kähler }\Longleftrightarrow \Delta_{BC}\omega=0    
\end{equation} \label{34} \vspace{-18pt}
\end{lem}
\begin{proof}
    ($\Longleftarrow$) If $\Delta_{BC}\omega=0$, then (\ref{47}) implies that $\partial\omega=\bar\partial\omega=0$, hence $\omega$ is Kähler.

    ($\Longrightarrow$) Now assume that $\omega$ is Kähler, so we need to show that $\omega\in\text{Ker}(\partial\bar\partial)^*$ again by (\ref{47}). By definition we have $*^2=\text{Id}$ on 2-forms, where $*$ is the Hodge star operator induced by the Kähler metric $\omega$. This yields :
    \begin{equation}
        (\partial\bar\partial)^*=\bar\partial^*\partial^*=-*\partial*(-*\bar\partial*)=-*\partial\bar\partial*
    \end{equation}
    This implies that :
    \begin{equation}
        (\partial\bar\partial)^*\omega=-*\partial\bar\partial*\omega=-\dfrac{1}{(n-1)!}*\partial\bar\partial\omega^{n-1}=0
    \end{equation}
    since $\bar\partial\omega^{n-1}=(n-1)\omega^{n-2}\wedge\bar\partial\omega=0$ by Kählerianity of $\omega$. Hence $\Delta_{BC}\omega=0$.
\end{proof}
We have then the following deformation result :
\begin{thm} Let $\sigma:\mathscr{X}\longrightarrow B$ be a holomorphic family of compact complex manifolds $X_t=\sigma^{-1}(t)$, with $t\in B$. If $X_0$ is $L^p$-Kähler hyperbolic for some $1\leq p\leq+\infty$, then the nearby fibers are $L^p$-SKT hyperbolic. \label{62}
\end{thm}
\begin{proof} Let $U$ be a neighborhood of $0\in B$ for which $\{X_t\}_{t\in U}$ are all Kähler manifolds, and let $\omega_0$ be an $L^p$-Kähler hyperbolic metric on $X_0$, i.e. $\widetilde{\omega}_0=d\eta$, where $\eta$ is a $1$-form satisfying the property ($F_p$) on some fundamental domain $D\Subset\widetilde{X}_0$. Let $\{\gamma_t\}_{t\in U}$ be a $\mathscr{C}^{\infty}$-family of Kähler metrics on $\{X_t\}_{t\in U}$ given by the Kodaira-Spencer theorem (Theorem \ref{45}) such that $\gamma_0=\omega_0$. Then the induced complex structure $\mathcal{J}_t$ gives the following decomposition for $\omega_0$ (seen as a 2-form) on the fiber $X_t$ :
\begin{equation}
    \omega_0=(\omega_0)^{2,0}_{\mathcal{J}_t}+(\omega_0)^{1,1}_{\mathcal{J}_t}+(\omega_0)^{0,2}_{\mathcal{J}_t}\hspace{3pt},\qquad\quad\forall t\in U
\end{equation}
As a generalization of Proposition 8 in \cite{kodaira1960deformations}, one can prove that the Hodge numbers on every fiber $X_t$ satisfy : $h^{r,s}(X_t)=h^{r,s}(X_0)$, for any $t\in U$ (see Theorem 2.6.4 in \cite{popovicinon} for the proof), then $\omega_t^{r,s}:=F_t^{r,s}\left((\omega_0)^{r,s}_{\mathcal{J}_t}\right)$ form $\mathscr{C}^\infty$-families of forms of bidegree ($r,s$) on $X_t$ (with $r+s=2$), where $F_t^{r,s}:\mathscr{C}^\infty_{r,s}(X)\longrightarrow\mathcal{H}^{r,s}_{BC}(X_t,\C)$ are the $L^2_{\gamma_t}$-orthogonal projections (that vary smoothly with respect to $t\in U$ by Theorem 5 in \cite{kodaira1960deformations}) onto the Bott-Chern harmonic spaces $\mathcal{H}^{r,s}_{BC}(X_t,\C)$. Then $\{\omega^{1,1}_t\}_{t\in U}$ is a new $\mathscr{C}^\infty$-family of Kähler metrics on $X_t$ (up to shrinking $U$ about 0 thanks to Lemma \ref{34}) such that $\omega_0^{1,1}=\omega_0$. On the other hand, the Hodge decomposition with respect to any fiber $X_t$ (with $t\in U$) holds :
\begin{equation}
    \{\omega_0\}_{DR}\in H^2(X_0,\C)\simeq H^2(X_t,\C)\simeq H^{2,0}_{BC}(X_t,\C)\oplus H^{1,1}_{BC}(X_t,\C)\oplus H^{0,2}_{BC}(X_t,\C) 
\end{equation}
Thus, for any $t\in U$ we have :
\begin{equation}
    \omega_0=\omega^{2,0}_t+\omega^{1,1}_t+\omega^{0,2}_t+u_t\hspace{2pt},\qquad u_t\in\mathscr{C}^\infty_2(X_t)\cap\text{Im\hspace{1pt}}d \label{23}
\end{equation}
Clearly, the family $\{u_t\}_{t\in U}$ is smooth, due to the smoothness of the other families of forms involved in this decomposition. So one can consider, for each $t\in U$, the $L^2$-minimal solution of the equation :
\begin{equation}
    dv_t=u_t\hspace{4pt},\qquad\quad t\in U
\end{equation}
This is given by the formula : $v_t=d_t^*\Delta_t^{-1}u_t$, for all $t\in U$. Then, this family varies smoothly with respect to $t$ again thanks to Theorem 5 in \cite{kodaira1960deformations}. Hence, the decomposition (\ref{23}) becomes : 
\begin{equation}
    \omega_0=\omega^{2,0}_t+\omega^{1,1}_t+\omega^{0,2}_t+dv_t\hspace{3pt},\qquad\quad v_t\in\mathscr{C}_1^\infty(X_t)
\end{equation}
Pulling back to the universal cover, we get :
\begin{equation}
    \widetilde{\omega}_0:=\pi^*\omega_0=\pi^*\omega^{2,0}_t+\pi^*\omega^{1,1}_t+\pi^*\omega^{0,2}_t+d\pi^*v_t=d\eta
    \label{36}
\end{equation}
where $\pi$ denotes the universal covering map $\widetilde{X_t}\longrightarrow X_t$ (by abuse of notation, since these fibers $X_t$ are diffeomorphic, hence $\widetilde{X}_t\underset{\mathscr{C}^\infty}{\simeq}\widetilde{X}_0$). By taking the $(1,1)$-part in (\ref{36}) we get :
\begin{equation}
    \widetilde{\omega}^{1,1}_t:=\pi^*_t\omega_t^{1,1}=\partial_t(\eta^{0,1}_t-\pi^*v^{0,1}_t)+\bar\partial_t(\eta^{1,0}_t-\pi^*v^{1,0}_t)\hspace{2pt},\qquad\forall t\in U
\end{equation}
Put $\alpha_t:=\eta^{0,1}_t-\pi^*v^{0,1}_t$ and $\beta_t:=\eta^{1,0}_t-\pi^*v^{1,0}_t$, then 
\begin{equation}
    \widetilde{\omega}^{1,1}_t=\partial_t\alpha_t+\bar\partial_t\beta_t \label{25}
\end{equation}
If $p=+\infty$, then $\alpha_t$ and $\beta_t$ satisfy the following inequalities for any $\gamma\in\pi_1(X_0)$ :
\begin{equation*}
\begin{cases}
    \lVert\alpha_t\rVert_{L^\infty_{\widetilde{\omega}^{1,1}_t}(\widetilde{X}_t)}\leq \lVert\eta_t^{0,1}\rVert_{L^\infty_{\widetilde{\omega}^{1,1}_t}(\widetilde{X}_t)}+\lVert\pi^*v_t^{0,1}\rVert_{L^\infty_{\widetilde{\omega}^{1,1}_t}(\widetilde{X}_t)}\leq C_t\lVert\eta\rVert_{L^\infty_{\widetilde{\omega}_0}(\widetilde{X}_0)}+\lVert v_t\rVert_{L^\infty_{\omega^{1,1}_t}(X)}<\infty \\
    \lVert\beta_t\rVert_{L^\infty_{\widetilde{\omega}^{1,1}_t}(\widetilde{X}_t)}\leq \lVert\eta_t^{1,0}\rVert_{L^\infty_{\widetilde{\omega}^{1,1}_t}(\widetilde{X}_t)}+\lVert\pi^*v_t^{1,0}\rVert_{L^\infty_{\widetilde{\omega}^{1,1}_t}(\widetilde{X}_t)}\leq C_t\lVert\eta\rVert_{L^\infty_{\widetilde{\omega}_0}(\widetilde{X}_0)}+\lVert v_t\rVert_{L^\infty_{\omega^{1,1}_t}(X)}<\infty
\end{cases}
\end{equation*}
where $X$ is the differentiable manifold that underlies all the Kähler fibers $X_t$, and the constant $C_t>0$ depends on the Kähler metric $\omega^{1,1}_t$ on the compact Kähler fiber $X_t$. Hence, the fibers $X_t$ are SKT hyperbolic.

If $1\leq p<+\infty$, then the following inequalities hold :
\begin{align*}
\begin{cases}
    \lVert\alpha_t\rVert_{L^p_{\widetilde{\omega}^{1,1}_t}(\gamma D)}&\leq \lVert\eta_t^{0,1}\rVert_{L^p_{\widetilde{\omega}^{1,1}_t}(\gamma D)}+\lVert\pi^*v_t^{0,1}\rVert_{L^p_{\widetilde{\omega}^{1,1}_t}(\gamma D)}\\ &\leq C'_t\lVert\eta\rVert_{L^p_{\widetilde{\omega}_0}(\gamma D)}+\lVert v_t\rVert_{L^\infty_{\omega^{1,1}_t}(X)}\text{Vol}^{\frac{1}{p}}_{\widetilde{\omega}^{1,1}_t}(X)\leq C''_t \\
    \lVert\beta_t\rVert_{L^p_{\widetilde{\omega}^{1,1}_t}(\gamma D)}&\leq \lVert\eta_t^{1,0}\rVert_{L^p_{\widetilde{\omega}^{1,1}_t}(\gamma D)}+\lVert\pi^*v_t^{1,0}\rVert_{L^p_{\widetilde{\omega}^{1,1}_t}(\gamma D)}\\ &\leq C'_t\lVert\eta\rVert_{L^p_{\widetilde{\omega}_0}(\gamma D)}+\lVert v_t\rVert_{L^\infty_{\omega^{1,1}_t}(X)}\text{Vol}^{\frac{1}{p}}_{\widetilde{\omega}^{1,1}_t}(X)\leq C''_t
\end{cases}
\end{align*}
where $C'_t>0$ depends on the Kähler metric $\omega^{1,1}_t$ on $X_t$, and hence $C''_t>0$ does not depend on $\gamma$. Hence, the fibers $X_t$ are $L^p$-SKT hyperbolic.
\end{proof}
\subsection{Deformation of Kähler hyperbolic manifolds}
In this section, we will restrict ourselves to the study of deformations of Kähler hyperbolic manifolds. Then, we have the following result :
\begin{thm} Let $X$ be a compact Kähler manifold. If $X$ admits a Kähler metric $\omega$ which is SKT hyperbolic, then $\omega$ is $L^2$-Kähler hyperbolic. \label{58}
\end{thm}
One key ingredient of the proof, is the \textbf{spectral gap} for Hodge-de Rham Laplace operator $\Delta$ in Theorem \ref{33} holds for the universal cover of Kähler manifolds carying Kähler SKT hyperbolic metrics. More precisely, let $X$ be a complete Kähler manifold with, and let $\omega$ be a Kähler metric on $X$ such that $\omega=\partial\alpha+\bar\partial\beta$, where $\alpha$ and $\beta$ are bounded $(0,1)$ and respectively $(1,0)$-forms on $X$. Then there exists a constant $\lambda=\lambda\left(\lVert\alpha\rVert_{L^\infty_{\omega}(X)},\lVert\beta\rVert_{L^\infty_{\omega}(X)}\right)>0$ such that for any $k\neq\dim_\C X=n$ (\cite{marouani2023skt}, Theorem 3.19) :
\begin{equation}
    \langle\Delta\psi,\psi\rangle\geq\lambda^2\lVert\psi\rVert^2_{L^2_\omega(X)}\hspace{4pt},\qquad\quad\forall\psi\in L^2_k(X) \label{66}
\end{equation}
This in particular implies that the spectrum of $\Delta$ is included in $[\lambda^2,+\infty[$, when acting on $L^2$-forms of degree $k\neq n$. As a consequence, $X$ has vanishing $L^2$-cohomology away from the middle degree, i.e. $\mathcal{H}^k_{\Delta}(X)=0$, for any $k\neq n$ (\cite{marouani2023skt}, Corollary 3.20). Moreover, the inequality (\ref{66}) holds also true for $L^2$-forms of degree $n$ which are orthogonal to the $\Delta$-harmonic $n$-forms. The proof goes the same as for Theorem 1.4.A in \cite{gromov1991kahler}.
\begin{rmk}
  A spectral gap holds also for the $\bar\partial$-Laplacian $\Delta'':=\bar\partial\bar\partial^*+\bar\partial^*\bar\partial$ and the $\partial$-Laplacian $\Delta':=\partial\partial^*+\partial^*\partial$ thanks to the classical \textbf{Bochner-Kodaira-Nakano} identity, or \textbf{BKN} for short (see \cite{demailly1997complex}, Corollary 6.5 for instance) : 
\begin{equation}
    \Delta_{\widetilde{\omega}}=2\Delta_{\widetilde{\omega}}'=2\Delta_{\widetilde{\omega}}'' 
\end{equation}  
\label{56}
with : $\text{Spec}\Delta'=\text{Spec}\Delta''\subset\left[\dfrac{\lambda^2}{2},+\infty\right[$ when acting on ($p,q$)-forms with $p+q\neq n$. \label{40}
\end{rmk} 
On a complete (Riemannian) manifold $X$, we have the following formula for the Laplacian acting on $L^2$-forms of degree $k$ (Lemma p.7 in \cite{gromov1991kahler}) :
\begin{equation}
    \langle\Delta\psi,\psi\rangle=\lVert d\psi\rVert^2_{L^2_k(X)}+\lVert d^*\psi\rVert^2_{L^2_k(X)}\hspace{3pt},\qquad\forall\psi\in L^2_k(X) \label{39}
\end{equation} 
This can be obtained by showing that $d^*$ is indeed the adjoint of $d$ when acting on the space of $L^2$-forms using \textbf{Gaffney's cutoff function} technique (Corollary p.6 in \cite{gaffney1954special}) or through the use of \textit{exhaustion functions} (Theorem 3.2, Chapter VII in \cite{demailly1997complex}). In particular, we have the following \textbf{weak Kodaira decompositon} (due to K. Kodaira, \cite{kodaira1949harmonic}) for any $k\in\{0,\dots,2n\}$ :
\begin{equation}
    L^2_k(X)=\mathcal{H}_\Delta^k(X)\oplus\overline{d(L^2_{k-1}(X))}\oplus\overline{d^*(L^2_{k+1}(X))} \label{38}
\end{equation}
\begin{rmk}
One can prove that a similar formula to (\ref{39}) holds for the $\partial$-Laplacian and the $\bar\partial$-Laplacian, as well as the decomposition (\ref{38}). \label{68}
\end{rmk}
Now, we will prove that the positive lower bound on the spectrum of the Laplacian in (\ref{66}), that we have on the universal cover of a compact Kähler manifold carrying a Kähler metric which is ($\widetilde{\partial+\bar\partial}$)-bounded, implies that $d$ and $d^*$ have closed images. To show this, we need the following ingredient from functional analysis :
\begin{lem} Let $\mathcal{H}_1$ and $\mathcal{H}_2$ be two Hilbert spaces, and let $T:\mathcal{H}_1\longrightarrow\mathcal{H}_2$ be a closed operator. Then the following are equivalent :
\begin{enumerate}
    \item $\text{Im}T$ is closed, i.e. $\text{Im}T=\overline{\text{Im}T}$.
    \item $\text{Im}T^*$ is closed.
    \item There exists a constant $C>0$ such that : $\lVert Tu\rVert_{\mathcal{H}_2}\geq C\lVert u\rVert_{\mathcal{H}_1}$, $\forall u\in\text{Dom}T\cap(\text{Ker}T)^\perp$. 
\end{enumerate} \label{67}
\end{lem}
The proof of this Lemma can be found in almost any functional analysis textbook, or see for instance Proposition 2.3 and Theorem 2.4 in \cite{ohsawa2015l2}.
\begin{prop} Let $X$ be a complete Kähler manifold equipped with a Kähler metric $\omega$ such that : $\omega=\partial\alpha+\bar\partial\beta$, for some bounded forms $\alpha$ and $\beta$ on $X$. Then : 
\begin{align}
    d:L^2_k(X)&\longrightarrow L^2_{k+1}(X) \\
    d^*:L^2_k(X)&\longrightarrow L^2_{k-1}(X)
\end{align}
have closed images in every degree $k$. \label{21}
\end{prop}
\begin{proof}
First, we are going to prove the closedness of the images of $d$ away from degree $n$. So let us take a sequence $\{\psi_\nu\}_{\nu\geq1}\subset d(L^2_{k}(X))$ that converges to some $\psi\in L^2_{k+1}(X)$, so for every $\nu\geq1$ we have $\psi_\nu=d\beta_\nu$ for some $\beta_\nu$ in $L^2_{k}(X)$. The $\beta_\nu$'s cannot be all closed unless $\psi=0$, so we can choose the $\beta_\nu$'s such that : $\beta_\nu\in(\text{Ker}d)^\perp\subset(\text{Im}d)^\perp=\text{Ker}d^*$. Hence $d^*\beta_\nu=0$, so (\ref{66}) together with (\ref{39}) imply :
\begin{equation*}
\lVert \psi_\nu\rVert_{L^2_{k+1}(X)}=\lVert d\beta_\nu\rVert_{L^2_{k+1}(X)}\geq \lambda\lVert \beta_\nu\rVert_{L^2_k(X)}
\end{equation*}
and this is equivalent to the closedness of $d(L^2_k(X))$ for every degree $k\neq n$ by Lemma \ref{67}. In a similar fashion, we can show that $d^*(L^2_k(X))$ is closed whenever $k\neq n$. When $k=n$, we can see that $d^*:L^2_n(X)\longrightarrow L^2_{n-1}(X)$ has a closed image due to Lemma \ref{67}, since it is the dual of $d:L^2_{n-1}(X)\longrightarrow L^2_n(X)$, which has a closed image as we proved previously. Similarly, $d:L^2_n(X)\longrightarrow L^2_{n+1}(X)$ has a closed image since it is the adjoint of the operator $d^*:L^2_{n+1}(X)\longrightarrow L^2_n(X)$, which has a closed image.
\end{proof} 
In particular, we have the following othogonal decompositions in every degree :
\begin{equation}
    \begin{cases}
    L^2_k(X)&=d(L^2_{k-1}(X))\oplus d^*(L^2_{k-1}(X))\hspace{4pt},\qquad \forall k\neq n \\
    L^2_n(X)&=\mathcal{H}^n_{\Delta}(X)\oplus d(L^2_{n-1}(X))\oplus d^*(L^2_{n-1}(X))  \label{20}
\end{cases}
\end{equation}
\begin{rmk}
Similarly, the images of $\partial$, $\partial^*$, $\bar\partial$ and $\bar\partial^*$ are also closed when acting on $L^2$-forms of any bidegree ($p,q$), due to Remark \ref{40}.  \label{50}
\end{rmk}
Now, we consider the following $\bar\partial$-type problem on a complete Kähler manifold $X$ endowed with a kähler metric $\omega=\partial\alpha+\bar\partial\beta$, for some bounded forms $\alpha$ and $\beta$ on $X$ :
\begin{equation}
\begin{cases}
\bar\partial u=\bar\partial v\hspace{2pt},\qquad v\in L^2_{k,0}(X)\hspace{2pt} \\
\partial u=0
\end{cases} \label{41}
\end{equation}
This kind of $\bar\partial$-problems has been treated in the compact Hermitian case by D. Popovici and S. Dinew (\cite{dinew2021generalised}, Lemma 2.4) by the use of the Bott-Chern Laplacian $\Delta_{BC}$. The ellipticity of $\Delta_{BC}$ induces the following \textit{weak} orthogonal Kodaira-type decomposition (which holds for any \textit{complete} Hermitian manifold) for any bidegree ($p,q$) :
\begin{equation}
    L^2_{p,q}(X)=\mathcal{H}_{BC}^{p,q}(X)\oplus\overline{\partial\bar\partial(L^2_{p-1,q-1}(X))}\oplus\overline{\partial^*(L^2_{p+1,q}(X))+\bar\partial^*(L^2_{p,q+1}(X))} \label{42}
\end{equation}
In order to get estimates on the solution of (\ref{41}), we recall some of the fundamental identities for the Bott-Chern Laplacian $\Delta_{BC}$ on Kähler manifolds. For instance, one can show that $\Delta_{BC}$ can be expressed in the following way (\cite{schweitzer2007autour}, Proposition 2.4) :
\begin{equation}
    \Delta_{BC}=(\Delta_{\widetilde{\omega}}'')^2+\partial_{\widetilde{\omega}}^*\partial_{\widetilde{\omega}}+\bar\partial_{\widetilde{\omega}}^*\bar\partial_{\widetilde{\omega}} \label{43}
\end{equation}
In particular : $\mathcal{H}_{BC}^{p,q}(X):=\text{Ker}\Delta_{BC}\cap L^2_{p,q}(X)\simeq\mathcal{H}_\Delta^{p+q}(X)\cap L^2_{p,q}(X)$. 

The formula (\ref{43}) implies that $\text{Spec}\Delta_{BC}\subset\left[\dfrac{\lambda^4}{4},+\infty\right[$, when $\Delta_{BC}$ acts on $L^2$-forms of bidegree ($p,q$) with $p+q\neq n$. Moreover, when $p+q=n$ we have :
\begin{equation}
    \langle\Delta_{BC}\psi,\psi\rangle\geq\dfrac{\lambda^4}{4}\lVert\psi\rVert^2_{L^2_{p,q}(X)}\hspace{4pt},\qquad\forall\psi\in(\mathcal{H}^{p,q}_{BC}(X))^\perp 
\end{equation}
In order to solve the $\bar\partial$-problem (\ref{41}), it is crucial to have a \textit{strong} Kodaira decomposition associated to the Bott-Chern Laplacian. Namely, we have :
\begin{prop}
    Let $X$ be a complete Kähler manifold equipped with a Kähler metric $\omega=\partial\alpha+\bar\partial\beta$, for some bounded forms $\alpha$ and $\beta$ on $X$. Then for any $p,q\in\{0,\dots,n\}$, $\partial\bar\partial(L^2_{p,q}(X))$ and $\partial^*(L^2_{p+1,q}(X))+\bar\partial^*(L^2_{p,q+1}(X))$ are closed.
\end{prop}
\begin{proof}
    By Lemma \ref{67}, the closedness of $\partial\bar\partial(L^2_{p,q}(X))$ is equivalent to :
    \begin{equation}
        \exists C>0\hspace{2pt},\quad\lVert \partial\bar\partial\psi\rVert_{L^2_{p+1,q+1}(X)}\geq C\lVert \psi\rVert_{L^2_{p,q}(X)}\hspace{4pt},\qquad\forall\psi\in L^2_{p,q}(X)\cap(\text{Ker}\partial\bar\partial)^\perp \label{69}
    \end{equation}
Let $\psi=\psi_1+\psi_2\in\text{Ker}\bar\partial+\text{Ker}\partial$, then we have :
\begin{equation*}
    \partial\bar\partial\psi=\partial\bar\partial\psi_1+\partial\bar\partial\psi_2=\partial(\bar\partial\psi_1)-\bar\partial(\partial\psi_2)=0
\end{equation*}
Hence : $\text{Ker}\bar\partial+\text{Ker}\partial\subset\text{Ker}\partial\bar\partial$, this implies that :
\begin{align}
    (\text{Ker}\partial\bar\partial)^\perp&\subset(\text{Ker}\bar\partial+\text{Ker}\partial)^\perp=(\text{Ker}\bar\partial)^\perp\cap(\text{Ker}\partial)^\perp \nonumber \\ &\subset(\text{Im}\bar\partial)^\perp\cap(\text{Im}\partial)^\perp=\text{Ker}\bar\partial^*\cap\text{Ker}\bar\partial^* \label{70}
\end{align}
So let $\psi\in L^2_{p,q}(X)\cap(\text{Ker}\partial\bar\partial)^\perp$. Then $\partial^*\psi=\bar\partial^*\psi=0$ by (\ref{70}). Hence we have :
\begin{align}
    \lVert\partial\bar\partial\psi\lVert^2_{L^2_{p+1,q+1}(X)}&\overset{(a)}{=}\lVert\partial\bar\partial\psi\lVert^2_{L^2_{p+1,q+1}(X)}+\lVert\partial^*\bar\partial\psi\lVert^2_{L^2_{p-1,q+1}(X)}\overset{(b)}{\geq}\dfrac{\lambda^2}{2}\lVert\bar\partial\psi\rVert^2_{L^2_{p,q+1}(X)}\\
    \lVert\bar\partial\psi\rVert^2_{L^2_{p,q+1}(X)}&\overset{(a')}{=}\lVert\bar\partial\psi\rVert^2_{L^2_{p,q+1}(X)}+\lVert\bar\partial^*\psi\rVert^2_{L^2_{p,q-1}(X)}\overset{(b')}{\geq}\dfrac{\lambda^2}{2}\lVert\psi\rVert^2_{L^2_{p,q}(X)}
\end{align}
where ($a$) follows from the fact $\partial^*\bar\partial=-\bar\partial\partial^*$ on a Kähler manifold and $\partial^*\psi=0$, and ($b$) is a consequence of the Remarks \ref{68} and \ref{40}. Similarly, ($a'$) comes from the fact that $\bar\partial^*\psi=0$, and ($b'$) can be deduced from the Remarks \ref{68} and \ref{40}. Hence (\ref{69}) is satisfied with constant $C=\dfrac{\lambda^2}{2}>0$, which means that $\partial\bar\partial(L^2_{p,q}(X))$ is closed.

As for the closedeness of $\partial^*(L^2_{p+1,q}(X))+\bar\partial^*(L^2_{p,q+1}(X))$, it follows from the fact that :
\begin{equation*}
    \partial^*(L^2_{p+1,q}(X))+\bar\partial^*(L^2_{p,q+1}(X))=d^*(L^2_{p,q}(X))=d^*(L^2_{p+q+2}(X))\cap(L^2_{p+1,q}(X)\oplus L^2_{p,q+1}(X))
\end{equation*}
This is closed since $d^*$ has closed images by Proposition \ref{21}, and $L^2_{p+1,q}(X)\oplus L^2_{p,q+1}(X)$ is a closed subspace of $L^2_{p+q+1}(X)$.
\end{proof}
This implies that we have the following \textit{strong} orthogonal decomposition in any bidegree ($p,q$) instead of the one in (\ref{42}) :
\begin{equation}
\begin{cases} 
   L^2_{p,q}(X)=\mathcal{H}_{BC}^{p,q}(X)\oplus\partial\bar\partial(L^2_{p-1,q-1}(X))\oplus(\partial^*(L^2_{p+1,q}(X))+\bar\partial^*(L^2_{p,q+1}(X))) \\
   \mathcal{H}_{BC}^{p,q}(X)=0\hspace{4pt},\qquad\quad \text{if }\hspace{4pt}p+q\neq n
\end{cases} \label{12}
\end{equation}
\begin{lem} The $L^2$-minimal norm solution of (\ref{41}) satisfies the following estimate :
\begin{equation}
    \lVert u\rVert^2_{L^2_{k,0}(X)}\leq2\left(1+\dfrac{4}{\lambda^4}\right)\lVert v\rVert^2_{L^2_{k,0}(X)}
     \label{57}
\end{equation} \vspace{-25pt}
\end{lem} 
\begin{proof} 
    The solution of (\ref{41}) is unique up to $\text{Ker}\partial\cap\text{Ker}\bar\partial=\text{Ker}\Delta_{BC}\oplus\text{Im}\partial\bar\partial$. Hence the decomposition (\ref{12}) allows us to get the following formula for the minimal $L^2$-norm solution of (\ref{41}) (see \cite{dinew2021generalised}, proof of Lemma 2.4) :
    \begin{equation*}
        u=\Delta_{BC}^{-1}(\bar\partial^*\bar\partial v+\bar\partial^*\partial\partial^*\bar\partial v)
    \end{equation*}
    where $\Delta_{BC}^{-1}$ is the \textbf{Green operator} associated to $\Delta_{BC}$, which is well-defined on $L^2_{p,0}(\widetilde{X})$ ; it is exactly the inverse of $\Delta_{BC}$ if $k\neq n$. If $k=n$, then $\Delta_{BC}^{-1}$ is the inverse of $\Delta_{BC}$ on $(\mathcal{H}^{n,0}_{BC}(X))^\perp$, which extend by 0 across $\mathcal{H}^{n,0}_{BC}(X)$. Hence we get :
\begin{align}
    u&=\Delta_{BC}^{-1}(\bar\partial^*\bar\partial v+\bar\partial^*\Delta'\bar\partial v) \label{52} \\ &=\Delta_{BC}^{-1}(\bar\partial^*\bar\partial v+\Delta''\bar\partial^*\bar\partial v) \label{53} \\ &=\Delta_{BC}^{-1}(\Delta'' v+(\Delta'')^2v) \label{54}
    \\ &=\Delta''\Delta_{BC}^{-1}v+(\Delta'')^2\Delta_{BC}^{-1}v \label{55} 
\end{align}
where (\ref{52}) comes from the fact that $\partial\bar\partial v=0$, as this condition is imposed on $v$ in order for (\ref{41}) to be solvable. (\ref{53}) is a consequence of the BKN identity in Remark \ref{56} and  $\Delta''\bar\partial^*=\bar\partial^*\Delta''$ on Kähler manifolds. (\ref{54}) comes from the fact that $\bar\partial^*v=0$ for bidegree reasons. As for (\ref{55}), it holds true since it is easy to check that $\Delta_{BC}$ commutes with both $\bar\partial^*\bar\partial$ and $\bar\partial\bar\partial^*$, and hence it commutes with $\Delta''$, which implies that $\Delta_{BC}^{-1}$ also commutes with $\Delta''$ thanks to the Kählerianity of $X$. Hence we get the following inequality :
\begin{equation}
    \lVert u\rVert^2_{L^2_{k,0}(X)}\leq2\left(\lVert\Delta''\Delta_{BC}^{-1}v\rVert^2_{L^2_{k,0}(X)}+\lVert(\Delta'')^2\Delta_{BC}^{-1}v\rVert^2_{L^2_{k,0}(X)}\right)
\end{equation}
Thanks to the formula (\ref{43}), a similar proof of (\ref{39}) yields for any $\psi\in L^2_{p,q}(X)$ :
\begin{equation}
    \langle\Delta_{BC}\psi,\psi\rangle=\lVert \partial\psi\rVert^2_{L^2_{p+1,q}(X)}+\lVert \bar\partial\psi\rVert^2_{L^2_{p,q+1}(X)}+\lVert \Delta''\psi\rVert^2_{L^2_{p,q}(X)}
\end{equation}
in any bidegree ($p,q$). Hence, we obtain the following estimates :
\begin{align}
  \lVert\Delta''\Delta_{BC}^{-1}v\rVert^2_{L^2_{k,0}(X)}&\leq\langle\Delta_{BC}\Delta_{BC}^{-1}v,\Delta_{BC}^{-1}v\rangle=\langle v,\Delta_{BC}^{-1}v\rangle\leq\frac{4}{\lambda^4}\lVert v\rVert^2_{L^2_{k,0}(X)} \\ \lVert(\Delta'')^2\Delta_{BC}^{-1}v\rVert^2_{L^2_{k,0}(X)}&\leq\langle\Delta_{BC}\Delta''\Delta_{BC}^{-1}v,\Delta''\Delta_{BC}^{-1}v\rangle=\langle v,(\Delta'')^2\Delta_{BC}^{-1}v\rangle \\ &\overset{(*)}{\leq} \lVert v\rVert_{L^2_{k,0}(X)}\lVert(\Delta'')^2\Delta_{BC}^{-1}v\rVert_{L^2_{k,0}(X)}\leq\lVert v\rVert^2_{L^2_{k,0}(X)}
\end{align}
where ($*$) comes from the use of the Cauchy-Schwarz inequality. Thus, we get the desired estimate (\ref{57}) for the $L^2$-norm of the $L^2$-minimal solution $u$ of (\ref{41}). 
\end{proof}
\begin{rmk} Let $A\subset\widetilde{X}$ be the support of $v$, then the $L^2$-minimal norm solution $u$ to the problem (\ref{41}) is also supported in $A$. Indeed, suppose $\text{supp}(u)\not\subset A$, then we take :
\begin{equation}
u'=\begin{cases}
    u \quad\text{on }A \\ 0\quad\text{on } \widetilde{X}\setminus A
\end{cases}
\end{equation}
Then $u'$ is clearly a solution to (\ref{41}). Moreover we have : $\lVert u'\rVert_{L^2_{k,0}(X)}\leq\lVert u\rVert_{L^2_{k,0}(X)}$, which forces $\lVert u'\rVert_{L^2_{k,0}(X)}=\lVert u\rVert_{L^2_{k,0}(X)}$ by the minimality of $u$. Hence $\text{supp}(u)\subset A$. \label{59}
\end{rmk}
The same results and estimates can be obtained for the following "conjugate" problem of the $\bar\partial$-equation (\ref{41}) :
\begin{equation}
\begin{cases}
\partial u=\partial v\hspace{2pt},\qquad v\in L^2_{0,l}(X)\hspace{2pt} \\
\bar\partial u=0 \label{60}
\end{cases} 
\end{equation}
Namely, one can show that the $L^2$-minimal norm solution satisfies :
\begin{equation}
\begin{cases}
u=\Delta_{BC}^{-1}(\partial^*\partial v+\partial^*\bar\partial\bar\partial^*\partial v)=\Delta'\Delta_{BC}^{-1}v+(\Delta')^2\Delta_{BC}^{-1}v \\
\lVert u\rVert^2_{L^2_{0,k}(X)}\leq2\left(1+\dfrac{4}{\lambda^4}\right)\lVert v\rVert^2_{L^2_{0,k}(X)}  \label{61}
\end{cases}
\end{equation}
Now we are able to prove Theorem \ref{58} :
\begin{pr1}
     Let $\omega$ be a Kähler metric on $X$ such that $\widetilde{\omega}=\partial\alpha+\bar\partial\beta$, where $\alpha$ and $\beta$ are bounded ($0,1$) and ($1,0$)-forms respectively on $\widetilde{X}$, in particular they satisfy the ($F_2$) property. That is, for any fundamental domain $D$ of $\widetilde{X}$ we have :
\begin{equation}
\begin{cases}
\exists\hspace{1pt}C_\alpha>0\hspace{1pt},\quad\underset{\gamma\in\pi_1(X)}{\sup}\lVert\alpha\rVert_{L^2_{\widetilde{\omega}}(\gamma D)}\leq C_\alpha \\
\exists\hspace{1pt}C_\beta>0\hspace{1pt},\quad\underset{\gamma\in\pi_1(X)}{\sup}\lVert\beta\rVert_{L^2_{\widetilde{\omega}}(\gamma D)}\leq C_\beta
\end{cases}
\end{equation}
Consider for any $\gamma\in\pi_1(X)$ the following :
\begin{equation}
\alpha_\gamma=\begin{cases}
    \alpha \quad\text{on }\gamma\overline{D} \\ 0\quad\text{elsewhere}
\end{cases} \quad,\qquad\qquad \beta_\gamma=\begin{cases}
\beta \quad\text{on }\gamma\overline{D} \\ 0\quad\text{elsewhere}
\end{cases}
\end{equation}
So we get : $\alpha=\underset{\gamma\in\pi_1(X)}{\sum}\alpha_\gamma$ and $\beta=\underset{\gamma\in\pi_1(X)}{\sum}\beta_\gamma$. Moreover, we have for any $\gamma\in\pi_1(X)$ :
\begin{equation}
    \lVert \alpha_\gamma\rVert_{L^2_{1,0}(\widetilde{X})}=\lVert \alpha\rVert_{L^2_{1,0}(\gamma D)}\leq C_\alpha \qquad\text{and}\qquad \lVert \beta_\gamma\rVert_{L^2_{1,0}(\widetilde{X})}=\lVert\beta\rVert_{L^2_{0,1}(\gamma D)}\leq C_\beta
\end{equation}
We consider for every $\gamma\in\pi_1(X)$, the $L^2$-minimal norm solutions $\alpha'_\gamma$ and $\beta'_\gamma$ for the problems (\ref{41}) and (\ref{60}) respectively with $v=\alpha_\gamma$ (resp. $v=\beta_\gamma$). Then Remark \ref{59} yields that $\text{supp}(\alpha_\gamma')\subseteq\gamma\overline{D}$ and $\text{supp}(\beta_\gamma')\subseteq\gamma\overline{D}$. Set $\alpha'=\underset{\gamma\in\pi_1(X)}{\sum}\alpha'_\gamma$ and $\beta'=\underset{\gamma\in\pi_1(X)}{\sum}\beta'_\gamma$, then :
\begin{equation}
\begin{cases}
\partial\alpha'=\partial\alpha\hspace{2pt} \\
\bar\partial\alpha'=0 \label{64}
\end{cases} \qquad\quad\text{and}\qquad\quad
\begin{cases}
\bar\partial\beta'=\bar\partial\beta \\
\partial\beta'=0
\end{cases} 
\end{equation} 
We set the $1$-form $\theta=\alpha'+\beta'$, then $\widetilde{\omega}=d\theta$ with :
\begin{align}
\lVert \theta\rVert_{L^2_1(\gamma D)}&\leq\lVert\alpha'\rVert_{L^2_{1,0}(\gamma D)}+\lVert\beta'\rVert_{L^2_{0,1}(\gamma D)}\\ &\leq2\left(1+\dfrac{4}{\lambda^4}\right)\left(\lVert\alpha\rVert_{L^2_{1,0}(\gamma D)}+\lVert\beta\rVert_{L^2_{0,1}(\gamma D)}\right)\qquad\text{(by (\ref{57}) and (\ref{61}))} \\
&\leq2\left(1+\dfrac{4}{\lambda^4}\right)(C_\alpha+C_\beta)=:C_\theta\quad,\qquad\forall\gamma\in\pi_1(X)
\end{align}
where $C_\theta>0$ does not depend on $\gamma$, hence $\theta$ satisfies the ($F_2$) property on $D$, which means that $X$ is $L^2$-Kähler hyperbolic. \vspace{-15pt}
\end{pr1}
\begin{flushright}
    $\square$
\end{flushright}
Combining Theorem \ref{62} with Theorem \ref{58}, we get our second deformation result :
\begin{cor} Let $\sigma:\mathscr{X}\longrightarrow B$ be a holomorphic family of compact complex manifolds $X_t=\sigma^{-1}(t)$, with $t\in B$. Fix an arbitrary reference point $0\in B$. If $X_0$ is Kähler hyperbolic, then the nearby fibers are $L^2$-Kähler hyperbolic. \label{9}
\end{cor}
\begin{rmk}
     Based on Observation 4.4 in \cite{popovici2013deformation}), one may conjecture the following :
\end{rmk} 
\begin{conj}
    Let $X$ be a compact complex surface. Then :
    \begin{center}
        $X$ is Kähler hyperbolic $\Longleftrightarrow$ $X$ is sG hyperbolic.
    \end{center}
\end{conj}
If this can be proved, Theorem \ref{10} would imply that Kähler hyperbolicity is \textbf{open} under holomorphic deformations in complex dimension 2. In higher dimensions, a case where we know that Kähler hyperbolicity is open under small deformations is the following :
\begin{thm}
    Let $\sigma:\mathscr{X}\longrightarrow B$ be a holomorphic family of compact complex manifolds $X_t=\sigma^{-1}(t)$, with $t\in B$. Fix an arbitrary reference point $0\in B$. If $X_0$ Kähler hyperbolic and the hodge number $h^{2,0}(X_0)=0$, then the nearby fibers $X_t$ are also Kähler hyperbolic. \label{72}
\end{thm}
\begin{proof}
    Since $X_0$ is Kähler, the fibers $X_t$ are also Kähler for $t$ close enough to $0$ by Theorem \ref{45} and $h^{2,0}(X_t)=h^{2,0}(X_0)=0$, which in turns implies that $h^{0,2}(X_t)=h^{0,2}(X_0)=0$ thanks to the Hodge symmetry. In particular, it follows from the Hodge decomposition that : $H^{1,1}(X_t,\R)\simeq H^{1,1}(X_0,\R)=H^2_{DR}(X,\R)$, where $X$ is the underlaying differentiable manifold. Which means that $\mathcal{K}_{X_t}$ is open in $H^2_{DR}(X,\R)$, and since $\mathcal{K}_{X_0}\cap V^2_{hyp}(X_0)\neq\emptyset$, we also have $\mathcal{K}_{X_t}\cap V^2_{hyp}(X_t)\neq\emptyset$ for $t$ close enough to $0$ thanks to Remark \ref{71}. Hence the fibers $X_t$ are also Kähler hyperbolic.
\end{proof}
Examples of Kähler hyperbolic manifolds that satisfy the condition $h^{2,0}(X)=0$ are the \textit{fake projective planes} (or \textbf{Mumford surfaces}), which are compact complex surfaces $X$ with the same cohomology as $\CP^2$ (see \cite{mumford1979algebraic}). B. Klingler and S.K. Yeung showed that these are arithmetic quotients of the hyperbolic ball $\B^2_\C$ (\cite{klingler2003rigidite}, \cite{yeung8integrality}), hence they are Kähler hyperbolic and $h^{2,0}(X)=0$. One can get a family of such examples in the following way : Let $X=S\times S$, where $S$ is a fake projective plane (then again $X$ is Kähler hyperbolic and $h^{2,0}(X)=0$, embed $X$ in some $\CP^N$ and consider $X_t$ to be a family of hyperplane sections of $X$. Then $\{X_t\}$ is a family of Kähler hyperbolic manifolds and $h^{2,0}(X_t)=0$ by the Lefschetz hyperplane theorem. These examples and Theorem \ref{72} were kindly communicated to the author by B. Claudon.

A consequence of Corollary \ref{9} is the deformation openness of the SKH property :
\begin{cor}
Let $\sigma:\mathscr{X}\longrightarrow B$ be a holomorphic family of compact complex manifolds $X_t=\sigma^{-1}(t)$, with $t\in B$. Fix an arbitrary reference point $0\in B$. If $X_0$ is SKH, then the nearby fibers are also SKH.
\end{cor}
\begin{proof}
    Since $X_0$ is SKH, then $X_0$ is Kähler hyperbolic. Corollary \ref{9} implies that the fibers $X_t$ are $L^2$-Kähler hyperbolic for $t$ sufficiently close to $0$. In particular, these fibers $X_t$ are Kähler $\widetilde{d}$-exact by Remark \ref{49}. Moreover, since $\pi_1(X_0)$ is Gromov hyperbolic by assumption, $\pi_1(X_t)$ is also Gromov hyperbolic by Remark \ref{51}. Thus, the fibers $X_t$ are SKH.
\end{proof}
We hope to find a way to solve the $\bar\partial$-problem (\ref{41}) and its conjugate problem (\ref{60}) with $L^\infty$-estimates. This will allow us to prove the following :
\begin{conj}
    Let $X$ be a compact Kähler manifold. If $X$ admits a Kähler metric $\omega$ which is SKT hyperbolic, then $X$ is Kähler hyperbolic.
\end{conj}
If this is proved to be true, we would end up showing that the Kähler hyperbolicity property is open under small deformations.

Maybe, a key assumption to prove this conjecture is to assume that the universal cover $\widetilde{X}$ of a Kähler manifold admitting a Kähler SKT hyperbolic metric is \textit{Stein}. In fact, we proved that such manifolds are $L^2$-Kähler hyperbolic (Theorem \ref{58}). In particular, they are Kähler $\widetilde{d}$-exact by Remark \ref{49}. Hence, they have generecally large (so infinite) fundamental groups by Corollary \ref{13}. These manifolds are also projective Theorem 2.8 in \cite{li2019kahler}.

The \textbf{Shafarevich uniformization conjecture} predicts that the universal cover should be Stein. To show that the Shafarevich conjecture holds true for these manifolds, or at least ensure the existence of non-constant holomorphic maps on the universal cover, we do not have a general method except the following \cite{eys} :
\begin{enumerate}
    \item If the fundamental group does not satisfy \textbf{Kazhdan's property (T)} (\cite{kazhdan1967connection}), N. Mok proved that the universal cover carries a non constant holomorphic function to some Hilbert space, equivariant for some aﬃne isometric representation of the fundamental group \cite{mok1995harmonic}. Consequently, the universal cover supports a non-constant holomorphic function with bounded diﬀerential, and thus, of at most linear growth (see \cite{delzant2005cuts}, p.6). But there are some Kähler hyperbolic manifolds that are known to satisfy this property (compact quotients of irreducible bounded symmetric domains of rank $2$).
    \item Find a finite-dimensional linear representation of the fundamental group of infinite image \cite{eyssidieux2012linear}. We do not know of any fundamental group of a smooth complex algebraic variety that we can prove is infinite, and that does not have such linear representations.
\end{enumerate}

\bibliographystyle{alpha}
\input{KH.bbl}

\end{document}

%% file: KH.bbl
\newcommand{\etalchar}[1]{$^{#1}$}